\documentclass[11pt,reqno]{amsart}
\usepackage{amsmath,amsfonts,amssymb,amsthm,amscd,comment,euscript}
\usepackage[all]{xy}
\usepackage{mathrsfs}
\usepackage{graphicx}
\usepackage{mathptmx}
\usepackage{enumerate} 
\usepackage[colorlinks=true]{hyperref} 
\usepackage[usenames,dvipsnames]{xcolor}
\usepackage{enumitem}

\xymatrixcolsep{1.9pc}                          
\xymatrixrowsep{1.9pc}
\newdir{ >}{{}*!/-5\pt/\dir{>}}                  

 \calclayout

\theoremstyle{plain}
\newtheorem{lemma}{Lemma}[section]
\newtheorem{corollary}[lemma]{Corollary}
\newtheorem{proposition}[lemma]{Proposition}
\newtheorem{theorem}[lemma]{Theorem}
\newtheorem*{thmm}{Theorem}
\newtheorem*{sublem}{Sublemma}

\theoremstyle{definition}

\newtheorem{exs}[lemma]{Examples}
\newtheorem{remark}[lemma]{Remark}
\newtheorem{defs}[lemma]{Definition}

\makeatletter

  \DeclareMathOperator{\Hom}{Hom}

\DeclareMathOperator{\End}{End} 
 
\DeclareMathOperator{\id}{id} 
 \DeclareMathOperator{\Gr}{Gr}
 
\DeclareMathOperator{\Ab}{Ab} \DeclareMathOperator{\fg}{fg}
\DeclareMathOperator{\fp}{fp} \DeclareMathOperator{\Qcoh}{Qcoh}
 \DeclareMathOperator{\kr}{Ker}
 
\DeclareMathOperator{\Ext}{Ext} \DeclareMathOperator{\im}{Im}
\DeclareMathOperator{\coker}{Coker} \DeclareMathOperator{\QGr}{QGr}
 \DeclareMathOperator{\Tors}{Tors}
 
\DeclareMathOperator{\Proj}{\sf Proj}

\DeclareMathOperator{\Mod}{Mod} 
 
 \DeclareMathOperator{\Map}{Map}
 \DeclareMathOperator{\Spec}{Spec}

\newcommand{\uhom}{\underline{\operatorname{Hom}}}
\renewcommand{\leq}{\leqslant}
\renewcommand{\geq}{\geqslant}
\newcommand {\lp}{\varinjlim}
\newcommand{\lo}{\varprojlim}
\newcommand{\lra}[1]{\bl{#1}\longrightarrow\relax}
\newcommand{\bl}[1]{\buildrel #1\over}
\newcommand{\cc}{\mathcal}
\newcommand{\ps}{\oplus}
\newcommand{\ff}{\mathfrak}

\newcommand{\ifff}{if and only if }

\newcommand{\bb}{\mathbb}

\DeclareMathOperator{\Pic}{Pic}
\DeclareMathOperator{\sgn}{sgn}
\newcommand{\PP}{\mathbb P}
\newcommand{\Aff}{\mathbb A}
\newcommand{\Z}{\mathbb Z}

\begin{document}

\footskip30pt


\title{Symmetric quasi-coherent sheaves}
\author{Grigory Garkusha}
\address{Department of Mathematics, Swansea University, Fabian Way, Swansea SA1 8EN, UK}
\email{g.garkusha@swansea.ac.uk}

\begin{abstract}
Using methods of stable homotopy theory, the category of symmetric quasi-co\-herent sheaves
associated with non-commutative graded algebras with extra symmetries is introduced and studied in this paper.
It is shown to be a closed symmetric monoidal Grothendieck category with invertible generators. It is proven that
the category of quasi-coherent sheaves on a projective scheme is recovered out 
of symmetric quasi-coherent sheaves. As an application, symmetric projective schemes
associated to such algebras are introduced and studied. It is shown that 
classical projective schemes are recovered from symmetric ones.
\end{abstract}



\maketitle
\thispagestyle{empty}
\pagestyle{plain}


\section{Introduction}

In his celebrated paper~\cite{Se}, Serre reconstructed the category of quasi-coherent sheaves
on a projective scheme out of graded modules. 
For a graded algebra $A$, the quotient category
$\QGr A$ retains Serre's description in a form that also applies to
noncommutative algebras and underlies noncommutative projective
geometry~\cite{AZ,S,V}. When $A$ is noncommutative, however, $\QGr A$
does not in general come equipped with a symmetric tensor product. The
grading shift remains an autoequivalence, but its interpretation as
tensoring with an invertible sheaf requires additional structure. The
purpose of this paper is to construct such a tensor structure for graded
algebras endowed with compatible symmetric-group actions, and to recover
classical projective geometry from the resulting theory.

The model for the construction comes from stable homotopy theory.
Symmetric spectra incorporate permutation actions into the levels of a
spectrum and provide a closed symmetric monoidal setting for
stabilization~\cite{HSS}. Our construction gives an abelian counterpart
of this procedure by localization in Gro\-then\-dieck categories. A related 
motivation comes from~\cite{GTLMS}, where
spectra are described using modules over a commutative symmetric ring
spectrum arising from Voevodsky's framed correspondences, building
on~\cite{GPJAMS}. Here we use the stabilization viewpoint to construct
categories of sheaves from symmetric graded algebras.


Let $k$ be a commutative ring. A commutative symmetric graded $k$-algebra
$E$ is a commutative algebra object in the category of symmetric
sequences of $k$-modules. Thus $E_n$ carries an action of $\Sigma_n$,
the multiplication maps are equivariant, and commutativity is expressed
using block permutations. The underlying graded algebra need not be
commutative in the ordinary sense. Ordinary commutative graded algebras
with trivial symmetric group actions form one class of examples. Further
examples arise from tensor algebras, exterior algebras, and algebras of
correspondences.

For $E_0=k$, let $\Gr^\Sigma E$ denote the category of symmetric
$E$-modules. This is a closed symmetric monoidal Grothendieck category
with projective generators $F_nE$, where $F_nE$ represents evaluation in
degree $n$. We define
\[
 \QGr^\Sigma E=\Gr^\Sigma E/\Tors^\Sigma E
\]
by a tensor localization in which the canonical evaluation morphisms
\[
 a_n:E[n]\wedge_EF_nE\longrightarrow E,\qquad n\geq0,
\]
become isomorphisms. Here $E[n]$ is the shifted symmetric module. These
maps express the compatibility between shifting and tensoring that is
required after stabilization. Their images identify $q(E[n])$ and
$q(F_nE)$ as mutually inverse objects. Our first main result records
the resulting monoidal structure.

\begin{thmm}\label{itogointr}
Let $E$ be a commutative symmetric graded $k$-algebra with $E_0=k$.
Then $\QGr^\Sigma E$ is a closed symmetric monoidal Grothendieck
category with a family of invertible generators
$\{\cc O^\Sigma(n)\}_{n\in\bb Z}$ satisfying
\[
 \cc O^\Sigma(m)\boxtimes\cc O^\Sigma(n)
 \cong\cc O^\Sigma(m+n).
\]
Moreover, $\QGr^\Sigma E$ is enriched locally finitely presented.
\end{thmm}

In particular, the twists are realized by tensoring with invertible
objects. These objects are strongly dualizable, and their internal Hom
functors preserve direct limits. This gives the enriched finiteness
statement without a finite-generation hypothesis on $E$.

The next result compares symmetric quasi-coherent sheaves with the
ordinary quotient category. It is the algebraic counterpart of the
comparison between spectra and symmetric spectra in topology.

\begin{thmm}[Reconstruction]\label{reconstrintr}
Suppose $R$ is a commutative finitely generated graded ring with $R_0$
being a $\bb Q$-algebra. Then there are equivalences of categories
   $$F:\QGr R\leftrightarrows\QGr^\Sigma R:U,$$
where $U$ is the forgetful functor. Here $R$ is regarded
as a commutative graded symmetric $R_0$-algebra with trivial action of the symmetric groups.

In particular, if $X=\Proj(R)$ is the projective scheme associated with a 
commutative graded ring $R$ with $R_0$ being a Noetherian
$\bb Q$-algebra such that $R$ is generated by finitely many elements of $R_1$,
then the composite functor
   $$\Qcoh X\lra{\Gamma_*}\QGr^{\bb Z}R\lra{I}\QGr R\lra{F}\QGr^\Sigma R$$
is an equivalence of categories, where $\Gamma_*(\cc F)=\bigoplus_{d=-\infty}^{+\infty}H^0(X,\cc F\otimes\cc O(d))$
is Serre's equival\-ence~\cite{Se} and $I$ is an equivalence of Theorem~\ref{equivqgr}.
\end{thmm}

The proof combines exactness of coinvariants over a $\bb Q$-algebra
with tensor invertibility in the localized category. In particular,
the symmetric groups act trivially on every closed symmetric
$R$-module. 
Thus the symmetric construction recovers the usual categories of
quasi-coherent sheaves on projective schemes under the stated
hypotheses.

We also associate a ringed space to $E$. Its points are the prime
$\Sigma$-ideals that do not contain $E_+$, and its topology is defined
similarly to Zariski's topology. We denote this space by $\Proj^\Sigma E$.
To construct its structure sheaf, we use endomorphism rings of the
tensor units in localizations associated with quasi-compact open
subsets. These endomorphism rings are commutative, so they define,
after sheafification, a sheaf of commutative rings
$\cc O_{\Proj^\Sigma E}$. The resulting ringed space is called the
symmetric projective scheme of $E$.

The following theorem says in particular that symmetric projective schemes share the same topological properties with classical ones.

\begin{thmm}
Suppose $E$ is finitely $\Sigma$-generated. Then the space $\Proj^\Sigma E$ is spectral, i.e.
it is $T_0$ and quasi-compact, the quasi-compact open
subsets are closed under finite intersections and form an open
basis, and every non-empty irreducible closed subset has a generic point. Moreover, if
$E$ is a commutative finitely generated graded ring with $E_0$
being a $\bb Q$-algebra, then the classical projective scheme $(\Proj E,\cc O_{\Proj E})$
can be identified with $(\Proj^\Sigma E,\cc O_{\Proj^\Sigma E})$. Here $E$ is regarded
as a commutative graded symmetric $E_0$-algebra with trivial action of the symmetric groups.
\end{thmm}

The symmetric theory also permits phenomena beyond classical
projective schemes. Appendix~\ref{appb} discusses the parity associated with the
self-symmetry of $\cc O^\Sigma(-1)$ and the distinction between the
underlying space, its structure sheaf, and the category $\QGr^\Sigma E$.
Over a field of characteristic zero, companion work~\cite{GG26}
constructs symmetric projective schemes that are locally ringed spaces
but are not schemes, as well as examples whose stalks are not local.
It also gives classical projective schemes as symmetric ringed spaces
whose symmetric quasi-coherent categories differ from the classical
ones. More generally, every quasi-projective $k$-scheme $U$ admits a
realization $\QGr^\Sigma E\simeq\Qcoh U$ with the underlying space a
projective compactification of $U$. These examples show that the
category $\QGr^\Sigma E$ and the associated ringed space record different aspects of
the symmetric algebra. Their detailed construction is deferred
to~\cite{GG26}.

The paper is organized as follows. Sections~\ref{prelim}
and~\ref{dubois} recall graded modules and torsion theories, and explain
the localization viewpoint on $\QGr A$. Section~\ref{symmsection}
develops symmetric graded modules. In Section~\ref{qgrsigma} we
construct $\QGr^\Sigma E$ and prove its tensor and finiteness
properties. Section~\ref{reconstrsection} proves the reconstruction
theorem. Section~\ref{sectionproj} constructs symmetric projective
schemes and establishes spectrality and the comparison with classical
projective schemes. Appendix~\ref{poss} recalls the required facts
about Grothendieck categories and Gabriel localization, and Appendix~\ref{appb}
discusses further properties and examples.

We write $[A,B]$ or $\uhom(A,B)$ for internal Hom in a closed
symmetric monoidal category, unless otherwise specified.
We regard $0$ as a natural number.

\section{Graded rings and modules}\label{prelim}

In this section we recall some basic facts about graded rings and
modules.

\begin{defs}{\rm
A {\it (positively) graded ring\/} is a ring $A$ together with a
direct sum decomposition $A=A_0\ps A_1\ps A_2\ps\cdots$ as abelian
groups, such that $A_iA_j\subset A_{i+j}$ for $i,j\geq 0$. A {\it
homogeneous element\/} of $A$ is simply an element of one of the
groups $A_j$, and a {\it homogeneous right ideal\/} of $A$ is a right ideal
that is generated by homogeneous elements. A {\it $\bb Z$-graded right
$A$-module\/} (respectively {\it graded right
$A$-module\/}) is an $A$-module $M$ together with a direct sum
decomposition $M=\ps_{j\in\bb Z}M_j$ (respectively $M=\ps_{j\in\bb N}M_j$) as abelian groups, such that
$M_jA_i\subset M_{j+i}$ for $i\geq 0,j\in\bb Z$ (respectively $j\in\bb N$). One calls $M_j$ the
$j$th {\it homogeneous component of $M$}. The elements $x\in M_j$
are said to be {\it homogeneous (of degree $j$)}.

Note that $A_0$ is a ring with $1\in A_0$, that all
summands $M_j$ are $A_0$-modules, and that $M=\ps_{j\in\bb Z}M_j$ (respectively $M=\ps_{j\in\bb N}M_j$) is
a direct sum decomposition of $M$ as an $A_0$-module.

Let $A$ be a graded ring. The {\it category of $\bb Z$-graded $A$-modules}
(respectively {\it graded $A$-modules}),
denoted by $\Gr^{\bb Z} A$ (respectively $\Gr A$), has as objects the $\bb Z$-graded $A$-modules
(respectively graded $A$-modules). A {\it
morphism\/} of graded $A$-modules $f:M\to N$ is an $A$-module
homomorphism satisfying $f(M_j)\subset N_j$ for all $j$. An
$A$-module homomorphism which is a morphism in $\Gr^{\bb Z}A$ (respectively $\Gr A$) is said to be
{\it homogeneous}. Note that $\Gr A$ is a full subcategory of $\Gr^{\bb Z}A$.

Let $M$ be a ($\bb Z$-)graded $A$-module and let $N$ be a submodule of $M$.
Say that $N$ is a {\it graded submodule\/} if it is a graded module
such that the inclusion map is a morphism in $\Gr A$. 
If $d$ is an
integer the {\it tail\/} $M_{\geq d}$ is the graded submodule of $M$
having the same homogeneous components $(M_{\geq d})_j$ as $M$ in
degrees $j\geq d$ and zero for $j<d$. We also denote the ideal
$A_{\geq 1}$ by $A_+$.
}\end{defs}

For $n\in\bb Z$, $\Gr A$ comes equipped with a shift functor
$M\mapsto M(n)$ where $M(n)$ is defined by $M(n)_j=M_{n+j}$. Then $\Gr^{\bb Z} A$ is a
Grothendieck category with a family of projective generators $\{A(n)\}_{n\in\bb
Z}$. Likewise, $\Gr A$ is a Grothendieck category with a family of 
projective generators $\{A(-n)\}_{n\in\bb N}$.

If $A$ is commutative, the tensor product for the category of all $A$-modules induces
a tensor product on $\Gr A$: given two ($\bb Z$-)graded $A$-modules $M,N$ and
homogeneous elements $x\in M_i,y\in N_j$, set $\deg(x\otimes
y):=i+j$. We define the {\it homomorphism $A$-module\/} $\cc
Hom_A(M,N)$ to be the graded $A$-module which is, in dimension $n$, the group $\cc
Hom_A(M,N)_n$ of graded $A$-module homomorphisms of
degree $n$, i.e.,
   $$\cc Hom_A(M,N)_n=\Gr^{\bb Z} A(M,N(n)).$$

We say that a ($\bb Z$-)graded $A$-module $M$ is {\it finitely generated\/} if
it is a quotient of a free graded module of finite rank
$\bigoplus_{s=1}^nA(d_s)$ for some $d_1,\ldots,d_s$. Say that $M$ is
{\it finitely presented\/} if there is an exact sequence
   $$\bigoplus_{t=1}^mA(c_t)\to\bigoplus
   _{s=1}^nA(d_s)\to M\to 0.$$
Note that any ($\bb Z$-)graded $A$-module is a direct
limit of finitely presented ($\bb Z$-)graded $A$-modules, and therefore $\Gr^{\bb Z} A$ and $\Gr
A$ are locally finitely presented Grothendieck categories.

A {\it  sh-torsion class\/} (`sh' for shift) in $\Gr^{\bb Z} A$ (respectively 
$\Gr A$) is a torsion class $\cc S\subset\Gr^{\bb Z} A$ (respectively $\cc S\subset\Gr A$) 
such that for any $X\in\cc S$ the shift $X(-n)$ is in $\cc S$ for any $n\in\bb Z$
(respectively $n\in\bb N$). By~\cite[Lemma~5.2]{GP2}
if $A$ is a commutative graded ring, then a torsion class $\cc S$ is sh-torsion 
in $\Gr^{\bb Z} A$ \ifff it is tensor, i.e. $X\otimes Y\in\cc S$ for any $X\in\cc S$ and $Y\in\Gr^{\bb Z}A$.

The following lemma was proven for commutative graded rings in~\cite[Lemma~5.3]{GP2}
(logically the reader should now read Appendix and then return to this section).

\begin{lemma}\label{Gabr}
Let $A$ be a graded ring. The map
   $$\cc S\longmapsto\ff F(\cc S)=\{\ff a\subseteq A\mid A/\ff a\in\cc S\}$$
establishes a bijection between the sh-torsion classes in $\Gr^{\bb Z} A$
and the sets $\ff F$ of homogeneous right ideals satisfying the
following axioms:
\begin{itemize}
\item[$T1.$] $A\in\ff F$;
\item[$T2.$] if $\ff a\in\ff F$ and $a$ is a homogeneous element of
$A$ then $(\ff a:a)=\{x\in A\mid ax\in\ff a\}\in\ff F$;
\item[$T3.$] if $\ff a$ and $\ff b$ are homogeneous right ideals of $A$ such that
    $\ff a\in\ff F$ and $(\ff b:a)\in\ff F$
    for every homogeneous element $a\in\ff a$ then $\ff b\in\ff F$.
\end{itemize}
We shall refer to such filters as {\em $t$-filters}. Moreover, $\cc
S$ is of finite type \ifff $\ff F(\cc S)$ has a basis of finitely
generated right ideals, that is every ideal in $\ff F(\cc S)$ contains a
finitely generated right ideal belonging to $\ff F(\cc S)$. In this case
$\ff F(\cc S)$ will be referred to as a {\em $t$-filter of finite
type}.
\end{lemma}

\begin{proof}
It is enough to observe that there is a bijection between the
Gabriel filters on the family $\{A(-n)\}_{n\in\bb Z}$ of projective generators
closed under the shift functor (i.e., if $\ff a$ belongs to the
Gabriel filter then so does $\ff a(-n)$ for all $n\in\bb Z$) and the
$t$-filters.
\end{proof}

The following statement was proven for commutative graded rings in~\cite[Proposition~5.4]{GP2}.

\begin{proposition}\label{eee}
The following statements are true:
\begin{enumerate}
\item Let $\ff F$ be a $t$-filter. If two-sided homogeneous 
ideals $I,J\subset A$ regarded as right ideals belong to $\ff F$, then
$IJ\in\ff F$.
\item Assume that $\ff B$ is a set of homogeneous two-sided ideals such that each ideal from 
$\ff B$ is finitely generated as a right ideal. Then the
set $\ff B'$ of finite products of two-sided ideals belonging to $\ff B$ is a
basis for a $t$-filter of finite type.
\end{enumerate}
\end{proposition}

\begin{proof}
(1). By definition, $IJ=\{\sum_ix_iy_i\mid x_i\in I,y_i\in J\}$.
For any homogeneous element $a\in I$ we have $(IJ:a)\supset J$,
so $IJ\in\ff F$ by $T3$ and the fact that every homogeneous right ideal
containing a right ideal from $\ff F$ must belong to $\ff F$.

(2). We follow~\cite[Proposition~VI.6.10]{St}. We must check that the set $\ff
F$ of homogeneous right ideals containing ideals in $\ff B'$ is a
$t$-filter of finite type. $T1$ is plainly satisfied. Let $a$ be a
homogeneous element in $A$ and $I\in\ff F$. There is an ideal
$I'\in\ff B'$ contained in $I$. Then $(I:a)\supset I'$ and therefore
$(I:a)\in \ff F$, hence $T2$ is satisfied as well.

Next we verify that $\ff F$ satisfies $T3$. Suppose $I$ is a
homogeneous right ideal and there is $J\in\ff F$ such that
$(I:a)\in\ff F$ for every homogeneous element $a\in J$. We may
assume that $J\in\ff B'$. Let $a_1,\ldots,a_n$ be generators of
$J$ as a right ideal. Then $(I:a_i)\in\ff F$, $i\leq n$, and $(I:a_i)\supset J_i$ for
some $J_i\in\ff B'$. It follows that $a_iJ_i\subset I$ for each $i$,
and hence $JJ_1\cdots J_n\subset J(J_1\cap\ldots\cap J_n)\subset I$,
so $I\in\ff F$.
\end{proof}

\section{Torsion modules and the category $\QGr A$}\label{dubois}

Let $A$ be a graded ring. In this section we introduce
the category $\QGr^{\bb Z}A$, which is analogous to the category of
quasi-coherent sheaves on a projective variety. It plays a prominent role in
``non-commutative projective geometry" (see, e.g., \cite{AZ,S,V}).

{\bf In the remainder of this section the homogeneous ideal
$A_+\subset A$ is assumed to be finitely generated.} 
Let $\Tors^{\bb Z}  A$ be the sh-torsion class of finite type in $\Gr^{\bb Z} A$ corresponding to the family of
homogeneous finitely generated ideals $\{A^n_+\}_{n\geq 1}$ (see
Proposition~\ref{eee}). Set $\Tors A=\Tors^{\bb Z}A\cap\Gr A$. Then $\Tors A$ is an sh-torsion class in $\Gr A$.
We call the objects of $\Tors^{\bb Z} A$ and $\Tors A$ {\it
torsion graded modules}. By construction, $M$ is torsion if and only if every 
homogeneous element of $M$ is annihilated by some power of $A_+$. As $A^n_+\subset A_{\geq n}$, $n\geq 1$,
then $A/A_{\geq n}$ is torsion. We also note that every bounded above graded 
module is torsion as its elements are annihilated by sufficiently large powers of $A_+$
(recall that a module $M$ is {\it bounded above\/} if $M_{\geq n}=0$ for some $n$).

Let $\QGr^{\bb Z} A=\Gr^{\bb Z} A/\Tors^{\bb Z} A$ and $\QGr A=\Gr A/\Tors A$. 
Let $Q:\Gr^{\bb Z}  A\to\QGr^{\bb Z} A$ and $q:\Gr A\to\QGr A$ denote the exact quotient functors.

The fully faithful embedding $i:\Gr A\to\Gr^{\bb Z}A$ is left adjoint to the truncation functor 
$\tau_{\geq 0}:\Gr^{\bb Z}A\to\Gr A$ sending $M$ to $M_{\geq 0}$. We have two functors
   $$\iota:\QGr A\hookrightarrow\Gr A\lra{i}\Gr^{\bb Z}A\lra{Q}\QGr^{\bb Z}A$$
and
   $$I:\QGr^{\bb Z}A\hookrightarrow\Gr^{\bb Z}A\xrightarrow{\tau_{\geq 0}}\Gr A\lra{q}\QGr A.$$

\begin{theorem}\label{equivqgr}
The functors $\iota:\QGr A\to\QGr^{\bb Z}A$ and $I:\QGr^{\bb Z}A\to\QGr A$ are mutually inverse 
equivalences of Gro\-then\-dieck categories.
\end{theorem}

\begin{proof}
Note that for any $M\in\QGr A$ the object $i(M)\in\Gr^{\bb Z}A$ has no torsion. Indeed, any morphism
$t:T\to i(M)$ with $T\in\Tors^{\bb Z}A$ induces a zero morphism $T_{\geq 0}\to M$ in $\Gr A$,
because $T_{\geq 0}\subset T$ is torsion and $M$ is torsionfree.
Since $i(M)$ is zero in negative degrees, it follows that $t=0$ and $i(M)$ is torsionfree.
Therefore the $\Tors^{\bb Z}A$-envelope $\lambda:i(M)\to\iota(M)=Q\circ i(M)$ is a monomorphism
with $T':=\iota(M)/i(M)\in\Tors^{\bb Z}A$.

We claim that $T'_{\geq 0}=0$. Indeed, we have a short exact sequence in $\Gr A$
   $$0\to M\to\tau_{\geq 0}\iota(M)\to T'_{\geq 0}\to 0.$$
As $T'_{\geq 0}\in\Tors A$ and $M$ is $\Tors A$-closed, this sequence is split.
We see that $\iota(M)_{\geq 0}=i(M)\oplus T'_{\geq 0}$. But $\iota(M)_{\geq 0}$ is a submodule of the torsionfree
module $\iota(M)$, hence $T'_{\geq 0}=0$ as claimed. The claim also implies $\iota(M)_{\geq 0}=i(M)$,
and so $I\circ\iota=\id$. In particular, $\iota$ is a faithful functor.

To show $\iota$ is full, consider $M,N\in\QGr A$ and a morphism $f:\iota(M)\to\iota(N)$ in $\Gr^{\bb Z} A$. 
We have a commutative diagram with exact rows
   $$\xymatrix{0\ar[r]&i(M)\ar[d]_{i(f_{\geq 0})}\ar[r]^{\lambda_M}&\iota(M)\ar[d]_f\ar[r]&T'\ar[r]\ar[d]&0\\
                       0\ar[r]&i(N)\ar[r]^{\lambda_N}&\iota(N)\ar[r]&S'\ar[r]&0}$$
As above $T',S'\in\Tors^{\bb Z}A$ and $T'_{\geq 0}=S'_{\geq 0}=0$. 
It follows that $\iota(f_{\geq 0})=Q\circ i(f_{\geq 0})=f$, and therefore $\iota$ is full.

Next, we prove that $\iota$ is essentially surjective. Every $M\in\Gr^{\bb Z}A$ fits in
a short exact sequence
   $$0\to M_{\geq 0}\to M\to M/M_{\geq 0}\to 0$$
As $M/M_{\geq 0}$ is bounded above, it is torsion. We see that
$Q(M_{\geq 0})\to Q(M)$ becomes an isomorphism
in $\QGr^{\bb Z}A$. The functor $Q\circ i$ takes $\lambda:M_{\geq 0}\to q(M_{\geq 0})$
to an isomorphism in $\QGr^{\bb Z} A$, because $\kr\lambda,\coker\lambda\in\Tors A$. If $M\in\QGr^{\bb Z} A$ then $q(M_{\geq 0})=I(M)$ 
and we have a zigzag of isomorphisms in $\QGr^{\bb Z} A$
    $$M=Q(M)\xleftarrow{\cong}Qi(M_{\geq 0})\xrightarrow{Qi(\lambda)} Qi(I(M))=\iota\circ I(M),$$
which is functorial in $M$. We see that $\iota$ is essentially surjective and $I$ is quasi-inverse to $\iota$. This completes the proof.
\end{proof}

Due to the preceding theorem we do not distinguish $\QGr^{\bb Z} A$ and $\QGr A$. From now on we 
will work with the category $\QGr A$. The construction of $\QGr A$ shares lots of common properties with
stable homotopy theory of spectra. To show this, we need some preparations. We introduce some terminology which
is inherited from stable homotopy theory, so that an expert in this field will be able to see what happens in the algebraic 
context around $\QGr A$.

Let $\ff G$ be the $t$-filter associated to $\Tors A$. By construction, powers of $A_+$ form a basis of $\ff G$.
Thus we have the following statement that follows from Proposition~\ref{eee}(2).

\begin{lemma}\label{fintype}
$\Tors A$ is a torsion theory of finite type and $\QGr A$ is a locally finitely generated
Grothendieck category.
\end{lemma}

It is also useful to have the following fact.

\begin{lemma}\label{tfilterG}
The following statements are true:

$(1)$ The $t$-filter $\ff G$ equals the smallest $t$-filter containing ideals $A_{\geq n}$, $n\geq 0$. In particular,
$\Tors A$ is the smallest sh-torsion class of $\Gr A$ containing $A/A_{\geq n}$, $n\geq 0$.

$(2)$ The $t$-filter $\ff G$ equals the smallest $t$-filter containing $A_{+}$.
In particular, $\Tors A$ is the smallest sh-torsion class of $\Gr A$ containing $A/A_{+}$.
\end{lemma}

\begin{proof}
(1). Denote by $\ff F$ the smallest $t$-filter containing ideals $A_{\geq n}$, $n\geq 0$. Then $A_+=A_{\geq 1}\in\ff F$,
and hence $A_+^n\in\ff F$ by Proposition~\ref{eee}(1). As the family $\{A_+^n\mid n\geq 0\}$ is a basis for $\ff G$,
it follows that $\ff G\subset\ff F$. On the other hand, each $A_{\geq n}$ is in $\ff G$, and hence $\ff F\subset\ff G$
due to the assumption that $\ff F$ is the smallest $t$-filter containing the ideals $A_{\geq n}$.

(2). The proof literally repeats that of (1).
\end{proof}

\begin{proposition}\label{qqq}
Let $\cc T$ be the full subcategory of $\Gr A$ consisting of the following modules $M$:
for every element $m\in M$ there is $n\in\bb N$ such that $mA_{\geq n}=0$. Then $\cc T$ is an
sh-torsion theory and it coincides with $\Tors A$. 
\end{proposition}

\begin{proof}
As \(A_+\) is generated as a right
ideal by homogeneous elements of degrees at most \(N\), then
\begin{equation}\label{contain}
A_{\geq N(r-1)+1}\subseteq A_+^r\subseteq A_{\geq r},
\quad r\geq1.
\end{equation}
Suppose $M\in\cc T$. Every $m\in M$ is annihilated by $A_{\geq n}$ for some $n$. Since $A_+^n\subset A_{\geq n}$,
the element $m$ is annihilated by $A_{+}^n$. We see that $\cc T\subset\Tors A$. On the other hand, if $M\in\Tors A$, then every element
$m\in M$ is annihilated by some power of $A_+$. The inclusions above imply that $m$ is annihilated by $A_{\geq n}$ for some $n$.
It follows that $\Tors A\subset\cc T$.
\end{proof}

The proof of the preceding proposition implies the following statement.

\begin{corollary}\label{qqqcor}
Suppose $A_{\geq 1}$ is generated by non-zero $b_1,\ldots,b_n$ and $N:=\max\{\deg(b_1),\ldots,\deg(b_n)\}$.
Then $\Tors A$ coincides with the full subcategory of $\Gr A$ consisting of the following modules $M$:
for every element $m\in M$ there is $n\in\bb N$ such that $mA_{\geq Nn}=0$.
\end{corollary}

If we regard injective objects as a kind of fibrant objects in homotopy theory,	
we say that an injective graded module $E\in\Gr A$ is
an {\it $\Omega$-module\/} if $\Gr A(-,E)$ takes each monomorphism in the family
(consisting of the shifts of the inclusion $A_+\hookrightarrow A$)
   \begin{equation}\label{Lambda}
    \Lambda=\{\lambda_n:A_+(-n)\to A(-n)\mid n\geq 0\},
   \end{equation}
to an isomorphism. If there is no likelihood of confusion,
we will sometimes call morphisms of graded $A$-modules becoming isomorphisms in $\QGr A$
{\it stable equivalences}. The {\it $m$-th shift $A$-module $M[m]$\/} of $M\in\Gr A$ is defined as $M[m]_n=M_{m+n}$.

\begin{lemma}\label{horosho}
$\Tors A$ coincides with the full subcategory which is left orthogonal to every 
$\Omega$-module. The injective objects of $\QGr A$ are the $\Omega$-modules. In particular,
a morphism $f:M\to N$ in $\Gr A$ is a stable equivalence if and only if $f^*:\Gr A(N,E)\to\Gr A(M,E)$
is an isomorphism for every $\Omega$-module $E$.
\end{lemma}

\begin{proof}
Consider a torsion theory 
   $$\cc S=\{M\in\Gr A\mid\Gr A(M,E)=0\textrm{ for every injective $\Omega$-module $E$}\}.$$
Since for every injective module $E\in\Gr A$ and $n\geq 0$ the $n$th shift $E[n]$ is injective, $\cc S$ 
must be sh-torsion.
As $(A/A_+)(-n)\in\cc S$, Lemma~\ref{tfilterG}(2) implies $\Tors A\subset\cc S$. 
On the other hand, every injective torsionfree module is an $\Omega$-module, and hence $\cc S\subset\Tors A$.
We also see that the injective objects of $\QGr A$ coincide with the injective $\Omega$-modules. 
\end{proof}

We shall write $\uhom_A(L,M)$ and 
$\underline{\Ext}_A^n(L,M)$ for $\bb N$-graded Abelian groups 
$\bigoplus_{m\geq 0}\Gr A(L,M[m])$ and $\bigoplus_{m\geq 0}\Ext^n_{\Gr A}(L,M[m])$.

It is also useful to have the following statement.

\begin{proposition}\label{vesma}
The following properties are equivalent:

\begin{enumerate}
\item $M\in\Gr A$ is $\Tors A$-closed;
\item $\uhom_A(A/A_{\geq n},M)=\underline{\Ext}_A^1(A/A_{\geq n},M)=0$ for all $n\geq 1$;
\item $\uhom_A(A/A_{+},M)=\underline{\Ext}_A^1(A/A_{+},M)=0$;
\item the functor $\uhom_A(-,M)$ takes $A_{\geq n}\to A$, $n\geq 1$, to an isomorphism;
\item the functor $\uhom_A(-,M)$ takes $A_{+}\to A$ to an isomorphism.
\end{enumerate}
\end{proposition}

\begin{proof}
Implications $(1)\Rightarrow(2)\Leftrightarrow(4)\Rightarrow(5)$ and 
$(1)\Rightarrow(2)\Rightarrow(3)\Leftrightarrow(5)$ 
are straightforward if we use the fact that
$A/A_{\geq n}\in\Tors A$ and $A$ is projective.

$(3)\Rightarrow(2)$. Suppose $\uhom_A(A/A_{\geq n},M)=\underline{\Ext}_A^1(A/A_{\geq n},M)=0$
for some $n$. By assumption, this is the case for $n=1$. One has a short exact sequence
   \begin{equation}\label{nado}
    0\to A_{\geq n}/A_{\geq n+1}\to A/A_{\geq n+1}\to A/A_{\geq n}\to 0.
   \end{equation}
As $A_{\geq n}/A_{\geq n+1}$ is covered by $\bigoplus_{x\in A_n}A/A_{\geq 1}(-n)$, we have 
a short exact sequence
    $$0\to\kr\alpha\to\oplus_{x\in A_n}A/A_{\geq 1}(-n)\xrightarrow{\alpha} A_{\geq n}/A_{\geq n+1}\to 0.$$
Since every element of $\bigoplus_{x\in A_n}A/A_{\geq 1}(-n)$ is annihilated by $A_{\geq 1}$, every element of
$\kr\alpha$ is annihilated by $A_{\geq 1}$ as well. It follows that $\kr\alpha$ is covered by
$\bigoplus_{x\in\kr\alpha}A/A_{\geq 1}(-n)$. Therefore,
   $$\uhom_A(\kr\alpha,M)=\uhom_A(A_{\geq n}/A_{\geq n+1},M)=
       \underline{\Ext}_A^1(A_{\geq n}/A_{\geq n+1},M)=0.$$
Applying $\underline{\Ext}_A^{0,1}(-,M)$ to~\eqref{nado}, our implication follows
by induction in $n$.

$(2)\Rightarrow(1)$. By Proposition~\ref{qqq} $T\in\Tors A$ if and only if every element of $T$ is annihilated
by $A_{\geq n}$ for some $n$. Therefore $T$ is covered by $\bigoplus_{x\in T}A/A_{\geq n_x}(-d_x)$,
where $d_x=\deg(x)$. Let $T':=\kr(\bigoplus_{x\in T}A/A_{\geq n_x}(-d_x)\twoheadrightarrow T)$, then 
$T'\in\Tors A$ and is covered by  $\bigoplus_{y\in T'}A/A_{\geq n_y}(-d_y)$.
The functor $\underline{\Ext}_A^{\geq 0}(-,M)$
commutes with coproducts. If we apply it to the short exact sequence
   $$T'\hookrightarrow\oplus_{x\in T}A/A_{\geq n_x}(-d_x)\twoheadrightarrow T,$$
we obtain $\uhom_A(T,M)=\underline{\Ext}_A^1(T,M)=0$
provided that $\uhom_A(A/A_{\geq n},M)=\underline{\Ext}_A^1(A/A_{\geq n},M)=0$.
Thus $M$ is $\Tors A$-closed.
\end{proof}

The proof of the preceding proposition and Corollary~\ref{qqqcor} imply the following statement.

\begin{corollary}\label{vesmacor}
Suppose $A_{\geq 1}$ is generated by $b_1,\ldots,b_m$ and $N:=\max\{\deg(b_1),\ldots,\deg(b_m)\}$.
Then $M\in\Gr A$ is $\Tors A$-closed if and only if 
$\uhom_A(A/A_{\geq Nn},M)=\underline{\Ext}_A^1(A/A_{\geq Nn},M)=0$ for all $n\geq 1$.
\end{corollary}

The torsion theory $\Tors A$ is the smallest localising subcategory of $\Gr A$ such that
each $\lambda_n$ from the family~\eqref{Lambda} becomes an isomorphism in the quotient category. 
In other words, if we start with the family~\eqref{Lambda},
we choose those injectives $E$ for which $\Gr A(-,E)$ takes each 
map in $\Lambda$ to an isomorphism (and call them $\Omega$-modules). 
$\Tors A$ is set to be the torsion theory of objects left orthogonal to $\Omega$-modules and $\QGr A:=\Gr A/\Tors A$.
Then each map in $\Lambda$ is a stable equivalence.

We have collected enough information about $\QGr A$ to illustrate its similarity with stable homotopy 
theory. Namely, consider the closed symmetric monoidal category
$\Gr(\bb S_*)$ of positively graded pointed simplicial sets. The 
sphere spectrum $\cc S=(S^0,S^1,S^2,\ldots)$, which is a kind of
the tensor algebra $T(S^1)$ of the circle $S^1$, is a (non-commutative) ring object in $\Gr(\bb S_*)$. Denote by 
$Sp_{S^1}$ the category of simplicial right $S^1$-spectra, which is nothing but the category of graded right modules 
in $\Gr(\bb S_*)$ over the sphere spectrum $\cc S$.

The category $Sp_{S^1}$ enjoys a level model structure in which weak equivalences and cofibrations are defined 
levelwise in each degree $n\in\bb N$. The stable model structure on $Sp_{S^1}$ is defined by Bousfield localization 
of the level model structure with respect to the family of natural monomorphisms
   $$\ff L=\{\ell_n:\cc S_+(-n)\to\cc S(-n)\mid n\geq 0\},$$
consisting of the shifts of the inclusion $\cc S_+\hookrightarrow\cc S$.
Here $\cc S_+=(*,S^1,S^2,\ldots)=S^1\wedge\cc S(-1)$ and the $-n$th shift $X(-n)$ of a spectrum $X=(X_0,X_1,\ldots)$
is the spectrum $(*,\bl{n-1}\ldots,*,X_0,X_1,\ldots)$. The fibrant objects in the stable model structure of $Sp_{S^1}$
are the $\Omega$-spectra. Recall that a map $f:X\to Y$ of spectra is a {\it stable equivalence\/} if it induces 
isomorphisms of stable homotopy groups. Equivalently, $f$ is a stable equivalence if and only if 
$f^*:\Map_*(Y,E)\to\Map_*(X,E)$ is a weak equivalence of simplicial sets for every $\Omega$-spectrum $E$.
Each map in $\ff L$ is then a stable equivalence and weak equivalences in the stable model structure are 
the stable equivalences.

We see that $\QGr A$ and $Sp_{S^1}$ can be defined in the same way as soon as we translate the algebraic and
topological languages into each other accordingly. It is worth mentioning that the role of Serre's
localization theory in Grothendieck categories is played by Bousfield's localization theory in model categories (and vice versa).

A disadvantage of $Sp_{S^1}$ is that it lacks the structure of a closed symmetric monoidal
category. In their breakthrough paper~\cite{HSS} Hovey, Shipley and Smith introduce and study
the category of symmetric spectra $Sp_{S^1}^\Sigma$. It is a closed symmetric monoidal
category and enjoys the stable model category structure which is Quillen equivalent to the stable 
model category of non-symmetric spectra $Sp_{S^1}$. Similarly to~\cite{HSS}, we can introduce and study
the category of ``symmetric quasi-coherent sheaves $\QGr^\Sigma A$" associated to (non-commutative)  graded algebras
enjoying a compatible action of the symmetric groups. It shares lots of common properties with symmetric spectra
and recovers classical quasi-coherent sheaves on a projective scheme (see Section~\ref{reconstrsection}).
 
\section{Symmetric graded modules}\label{symmsection}

In this section we fix a commutative ring $k$. If $G$ is a group we can identify left and right $kG$-modules
by means of the anti-automorphism $g\mapsto g^{-1}$ of $G$. In particular, a right $kG$-module $M$ 
is regarded as a left $kG$-modules if we set $gm:=mg^{-1}$ for all $m\in M$ and $g\in G$~\cite[p.~55]{Br}.
Consider a $k$-linear category $k\Sigma$ whose objects are the natural numbers. The
morphisms $k\Sigma(m,n)$ are the groups rings $k\Sigma_n$ whenever $m=n$ and zero otherwise.
The {\it category of symmetric sequences of $k$-modules $\Gr^\Sigma k$} is the category of additive functors
$(k\Sigma,\Ab)$. Denote each representable functor of $(k\Sigma,\Ab)$ by $F_mk$, $m\geq 0$. They form
a family of finitely generated projective generators of $(k\Sigma,\Ab)$. 

Unravelling the latter definition, a symmetric sequence in $\Gr^\Sigma k$ is a sequence $(M_0,M_1,\ldots)$
with each $M_n$ being a $k\Sigma_n$-module. Morphisms are graded morphisms of $k$-modules 
$(f_i:M_i\to N_i)_{i\in\bb N}$ such that each $f_i$ is a $k\Sigma_i$-homomorphism. 
The category $\Gr^\Sigma k$ is Grothendieck with a 
family of projective generators given by representable functors. As $k\Sigma$ is a symmetric
monoidal preadditive category (the monoidal product in $k\Sigma$ of $m,n\in\bb N$ is given by $m+n$), 
a theorem of Day~\cite{Day}
implies $\Gr^\Sigma k$ is closed symmetric monoidal with monoidal product 
   $$(M\wedge N)_t=\bigoplus_{p+q=t}k\Sigma_t\otimes_{k\Sigma_p\otimes k\Sigma_q}M_p\otimes N_q.$$
The monoidal unit is given by the symmetric sequence $(k,0,0,\ldots)$. By construction, $(F_mk)_n=(0,\ldots,k\Sigma_n,0,\ldots)$
and $F_mk\wedge F_nk=F_{m+n}k$.
   
The {\it tensor product\/} $f\wedge g:X\wedge Y\to X'\wedge Y'$ of the maps
$f:X\to X'$ and $g:Y\to Y'$ in $\Gr^\Sigma k$ takes $\alpha\otimes X_p\otimes Y_q$ to
$\alpha\otimes X_p'\otimes Y_q'$ by means of $f_p\otimes g_q$ for the summand labelled by
$\alpha\in k\Sigma_{p+q}$.

The twist isomorphism $twist:X\wedge Y\to Y\wedge X$ for $X,Y\in\Gr^\Sigma k$ is the
natural map taking the summand $\alpha\otimes X_p\otimes Y_q$ to the summand 
$\alpha\chi_{q,p}\otimes Y_q\otimes X_p$
for $\alpha\in k\Sigma_{p+q}$, where $\chi_{q,p}\in\Sigma_{p+q}$ is the $(q,p)$-shuffle
given by $\chi_{q,p}(i) = i+p$ for $1\leq i \leq q$ and $\chi_{q,p}(i) = i-q$ for $q < i \leq p + q$.
It is worth noting that the map defined without the shuffle permutation is not a map of symmetric sequences.

\begin{defs}\label{symmdef}
A {\it symmetric graded $k$-algebra\/} or a {\it $(\Sigma,k)$-algebra\/} $E$ is a ring 
object in $\Gr^\Sigma k$. In detail, $E$ consists of the following data:

\begin{itemize}
\item[$\diamond$] a sequence $E_n\in k\Sigma_n\Mod$ for $n\geq 0$;
\item[$\diamond$] $\Sigma_n\times\Sigma_m$-equivariant {\it multiplication maps}
   $$\mu_{n,m}:E_n\otimes E_m\to E_{n+m},\quad n, m\geq 0$$
\item[$\diamond$] a {\it unit map} $\iota_0:k\to E_0$.
\end{itemize}
This data is subject to the following conditions:

(Associativity) The square
   $$\xymatrix{E_n\otimes E_m\otimes E_p\ar[d]_{\mu_{n,m}\otimes\id}\ar[r]^{\id\otimes\mu_{m,p}}&E_n\otimes E_{m+p}\ar[d]^{\mu_{n,m+p}}\\
                       E_{n+m}\otimes E_{p}\ar[r]_{\mu_{n+m,p}}&E_{n+m+p}}$$
commutes for all $n, m, p\geq 0$.

(Unit) The two composites
   $$E_n\cong E_n\otimes k\xrightarrow{E_n\otimes\iota_0}E_n\otimes E_0\xrightarrow{\mu_{n,0}}E_n$$
   $$E_n\cong k\otimes E_n\xrightarrow{\iota_0\otimes E_n}E_0\otimes E_n\xrightarrow{\mu_{0,n}}E_n$$
are the identity for all $n\geq 0$. Note that $E_0$ is a unital $k$-algebra in this case.

A morphism $f:E\to E'$ of symmetric graded $k$-algebras consists of $k\Sigma_n$-homomorphisms
$f_n:E_n\to E'_n$ for $n\geq 0$, which are compatible with the multiplication and unit maps in the sense that
$f_{n+m}\circ\mu_{n,m}=\mu_{n,m}\circ(f_n\otimes f_m)$ for all $n,m\geq 0$, and $f_0\circ\iota_0=\iota_0$.

A symmetric graded $k$-algebra $E$ is {\it commutative\/} if the square
   $$\xymatrix{E_n\otimes E_m\ar[d]_{\mu_{n,m}}\ar[r]^{twist}&E_m\otimes E_{n}\ar[d]^{\mu_{m,n}}\\
                       E_{n+m}\ar[r]_{\chi_{n,m}}&E_{m+n}}$$
commutes for all $n,m\geq 0$. Note that $E_0$ is a commutative ring in this case.
\end{defs}

\begin{exs}\label{primery}
(1) The tensor algebra $T(V)=k\oplus V\oplus V^{\otimes 2}\oplus\cdots$ of a $k$-module $V$
over a commutative ring $k$ is a 
typical example of a symmetric commutative graded $k$-algebra. Each $\Sigma_n$ acts on $V^{\otimes n}$
by permutation. Note that $T(V)$ is not commutative as a non-symmetric graded $k$-algebra.

(2) The exterior $k$-algebra $\Lambda(V)=\bigoplus_{n\geq 0}\Lambda^nV$ of a $k$-module $V$
is commutative symmetric graded. By definition, $\Lambda^nV=V^{\otimes n}/\cc A^n$ with $\cc A^n$ generated by the elements
of the form $v_1\otimes\cdots\otimes v_n$ such that $v_i=v_j$ for some $i\ne j$. 
Each $\Sigma_n$ acts on $\Lambda^{n}V$ by permutation.

(3) Any graded commutative $k$-algebra $A=\bigoplus_{n\geq 0}A_n$ can 
be regarded as a symmetric commutative graded $k$-algebra with trivial action
of symmetric groups.

(4) Following~\cite[Section~2]{GTLMS} a class of commutative symmetric $k$-algebras can be constructed as follows.
Let $(\cc C,\otimes)$ be a symmetric monoidal $k$-linear category with finite colimits and let $E$ be a ring 
object in the monoidal category $\cc C^\Sigma$ of symmetric sequences in $\cc C$. For any object $P\in\cc C$
the graded $k$-module
   $$(E_0,\Hom_{\cc C}(P,E_1),\Hom_{\cc C}(P^{\otimes 2},E_2),\ldots)$$
is a symmetric $k$-algebra which is commutative whenever $E$ is.

(5) Given a symmetric motivic Thom $T$-spectrum $E$ in the sense of~\cite{GN} the graded Abelian group
of linear $E$-framed correspondences from point to point
   $$(\bb ZF^E_0(pt,pt),\bb ZF^E_1(pt,pt),\bb ZF^E_2(pt,pt),\ldots)$$
is a symmetric ring which is commutative whenever the motivic Thom ring spectrum $E$ is.

(6) Consider $k\Sigma_*:=\bigoplus_{n\geq 0}k\Sigma_n$ and let
   $$\mu_{n,m}:k\Sigma_n\otimes k\Sigma_m\to k\Sigma_{n+m}$$
be induced by $\Sigma_n\times \Sigma_m\to\Sigma_{n+m}$, $(\sigma,\tau)\mapsto\sigma\times\tau$.
The map $\iota_0:k\to k\Sigma_0$ is defined in a canonical way. With these morphisms $k\Sigma_*$
becomes a commutative symmetric $k$-algebra if we assume that $\Sigma_n$ acts by conjugation in degree $n$.

(7) Define
\[
E_{r,s}=k[t_1,\ldots,t_r]\otimes_k\Lambda_k(\theta_1,\ldots,\theta_s),
\qquad |t_i|=2,\quad |\theta_a|=1,
\]
and give \(E_n\) the action
\[
\sigma\cdot a=\operatorname{sgn}(\sigma)a.
\]
With these operations $E_{r,s}$ becomes a commutative symmetric $k$-algebra. If $s\geq 2$ and $\textrm{char}\, k\ne 2$, it is not commutative 
as an ordinary ungraded algebra, since $\theta_1\theta_2=-\theta_2\theta_1\ne0$.
\end{exs}

\begin{defs}\label{emodule}
A {\it right module $M$\/} over a symmetric graded $k$-algebra $E\in\Gr^\Sigma k$ consists of the following data:
\begin{itemize}
\item[$\diamond$] a sequence of objects $M_n\in k\Sigma_n\Mod$ for $n\geq 0$,
\item[$\diamond$] $\Sigma_n\times\Sigma_m$-equivariant action maps 
$\alpha_{n,m}:M_n\otimes E_m\to M_{n+m}$ for $n,m\geq 0$.
\end{itemize}
The action maps have to be associative and unital in the sense that the diagrams commute
   $$\xymatrix{M_n\otimes E_m\otimes E_p\ar[r]^{M_n\wedge\mu_{m,p}}\ar[d]_{\alpha_{n,m}\otimes E_p}&M_n\otimes E_{m+p}\ar[d]^{\alpha_{n,m+p}}
                       &&&M_n\cong M_n\otimes k\ar[dr]_{\id_{M_n}}\ar[r]^{M_n\otimes\iota_0}&M_n\otimes E_0\ar[d]^{\alpha_{n,0}}\\
                       M_{n+m}\otimes E_p\ar[r]_{\alpha_{n+m,p}}&M_{n+m+p}
                       &&&&M_n}$$
for all $n,m,p\geq 0$. A {\it morphism\/} $f:M\to N$ of right $E$-modules consists of
$\Sigma_n$-equivariant maps $f_n:M_n\to N_n$ for $n\geq 0$, which are compatible with the action
maps in the sense that $f_{n+m}\circ\alpha_{n,m}=\alpha_{n,m}\circ(f_n\otimes E_m)$ for all
$n,m\geq 0$. We denote the category of right $E$-modules by $\Gr^\Sigma E$.
\end{defs}

The category $\Gr^\Sigma E$ can be identified with the category of additive functors
from a $k$-linear category $\cc E$ to Ab. By definition, the objects of $\cc E$ are the natural numbers.
Its morphisms are given by $k$-modules
   $$\cc E(m,n)=  
   \begin{cases}
  k\Sigma_n\otimes_{k\Sigma_{n-m}}E_{n-m}, & m\leq n \\
  0, & m> n
  \end{cases}$$
Here $\Sigma_{n-m}$ is embedded into $\Sigma_n$ by $\sigma\mapsto\id_{\{1,\ldots,m\}}\times\sigma$.
Under this identification representable functors $\cc E(m,-)$ are mapped to
projective generators of $\Gr^\Sigma E$, denoted by $F_mE$. By construction,
$F_mE_n=k\Sigma_n\otimes_{k\Sigma_{n-m}}E_{n-m}$ if $m\leq n$ and zero otherwise.

\begin{remark}
If $E=T(V)$ is the tensor algebra of $V\in\Mod k$, then the category $\Gr^\Sigma E$
is identified with the category of $V$-spectra in the closed symmetric monoidal category $\Mod k$.
We refer the reader to~\cite[Section~7]{H} for a general context of $T$-spectra in reasonable categories.
\end{remark}

Given a symmetric commutative graded $k$-algebra $E$, the category $\Gr^\Sigma E$ is
closed symmetric monoidal with $E=F_0E$ the monoidal unit. The monoidal product
will be denoted by $\wedge_E$. By construction, $F_mE\wedge_E F_nE=F_{m+n}E$
for all $m,n\geq 0$. It is also worth mentioning that $F_mE=F_mk\wedge E$,
where the wedge product is taken in $\Gr^\Sigma k$.

\begin{lemma}\label{exact}
Suppose $E$ is a symmetric commutative graded $k$-algebra. Each projective generator $F_mE\in\Gr^\Sigma E$,
$m\geq 0$, is flat in the sense that $F_mE\wedge_E-$ is an exact functor.
\end{lemma}

\begin{proof}
Consider a monomorphism $f:M\to N$ in $\Gr^\Sigma E$. Then $F_mE\wedge_Ef$ is isomorphic to
$F_mk\wedge E\wedge_EM\to F_mk\wedge E\wedge_EN$. The latter arrow is isomorphic to
$\phi:F_mk\wedge M\to F_mk\wedge N$. In each level $\ell\geq 0$, the map $\phi_\ell$ equals
   $$\bigoplus_{n+q=\ell}k\Sigma_\ell\otimes_{k\Sigma_n\otimes k\Sigma_q}k\Sigma_n\otimes M_q\to
       \bigoplus_{n+q=\ell}k\Sigma_\ell\otimes_{k\Sigma_n\otimes k\Sigma_q}k\Sigma_n\otimes N_q.$$
By~\cite[Proposition~III.5.1]{Br} the latter arrow is isomorphic to
   $$\oplus_{g\in\Sigma_\ell/\Sigma_q}M_q\to\oplus_{g\in\Sigma_\ell/\Sigma_q}N_q,$$
where $g$ ranges over any set of representatives for the left cosets of $\Sigma_q$ in $\Sigma_\ell$.
As this map is a monomorphism, so is $\phi_\ell$. We see that $F_mE\wedge_Ef$ is degreewise a monomorphism,
and hence the functor $F_mE\wedge_E-$ is left exact. It is right exact as a left adjoint to the internal Hom-functor
$[F_mE,-]:\Gr^\Sigma E\to\Gr^\Sigma E$.
\end{proof}

The {\it shift symmetric $E$-module\/} $M[1]$ of $M\in\Gr^\Sigma E$ is defined as 
$M[1]_n=M_{1+n}$ with action
of $\Sigma_n$ by restriction of the $\Sigma_{1+n}$-action on
$M_{1+n}$ along the obvious embedding $\Sigma_n\hookrightarrow
\Sigma_{1+n}$ taking $\tau\in\Sigma_n$ to
$1\oplus\tau\in\Sigma_{1+n}$. The multiplication map $M[1]\wedge E\to M[1]$ is the
reindexed multiplication map for $M$. The $m$-th shift $M[m]$ is defined recursively:
$M[m]:=(M[m-1])[1]$. Note that $M[m]=[F_mE,M]$.

\begin{corollary}\label{inj}
Suppose $E$ is a symmetric commutative graded $k$-algebra and $Q$ is an injective $E$-module.
Then the $m$-th shift $Q[m]$ is injective as well. 
\end{corollary}

\begin{proof}
This follows from the preceding lemma and the fact that $Q[m]=[F_mE,Q]$.
\end{proof}

We recall the definition of the internal Hom-object $\uhom^\Sigma_E(M,N)\in\Gr^\Sigma E$ of $M,N\in\Gr^\Sigma E$
for the convenience of the reader (we also write it as $[M,N]$ for brevity).

First, recall that $\Gr^\Sigma E$ is tensored over $\Mod k$ as follows. For any $V\in\Mod k$ and $M\in\Gr^\Sigma E$
$V\otimes M$ is the symmetric sequence $(V\otimes M_0,V\otimes M_1,\ldots)$, where $\Sigma_n$ 
naturally acts on $V\otimes M_n$. Maps defining the $E$-module structure on $V\otimes M$ are given by
   $$(V\otimes M_n)\otimes E_l\cong V\otimes(M_n\otimes E_l)\xrightarrow{V\otimes m}V\otimes M_{n+l}.$$

Second, for any $k\geq 0$ there is a natural map $\phi^k:E_k\otimes M\to M[k]$ in $\Gr^\Sigma E$ given by
   $$\phi_n^k:E_k\otimes M_n\lra{tw}M_n\otimes E_k\lra{m}M_{n+k}\xrightarrow{\chi_{n,k}}M_{k+n}.$$
It is worth noting that this cannot be a morphism in $\Gr^\Sigma E$ without shuffle permutations.

Finally, we define a symmetric spectrum $[X,Y]$ in level $n$ by
   $$[M,N]_n:=\Gr^\Sigma E(M,N[n]).$$
The $\Sigma_n$-action is induced by the action on $N[n]$. Namely, in every level $N[n]_m$ 
the symmetric group $\Sigma_n$ acts by the inclusion $(-\times 1): \Sigma_n\to \Sigma_{n+m}$, 
and these actions are compatible with the structure multiplication maps. The structure multiplication map 
$[M,N]_n\otimes E_\ell\to[M,N]_{n+\ell}$ is the composite
   $$(M,N[n])\otimes E_\ell\to(M,E_\ell\otimes N[n])\xrightarrow{(M,\phi^\ell)}(M,N[\ell+n])\xrightarrow{(M,\chi_{\ell,n})}(M,N[n+\ell]).$$

With this description a natural isomorphism $[F_mE,M]\cong M[m]$ is given at level $n$ by
   $$[F_mE,M]_n=(F_mE,M[n])\cong M[n]_m=M_{n+m}\xrightarrow{\chi_{n,m}}M_{m+n}=M[m]_n.$$
We also refer the reader to~\cite[Example~I.2.25]{Sch} for the discussion of the internal Hom-object
on symmetric $S^1$-spectra.

\section{The category of symmetric quasi-coherent sheaves}\label{qgrsigma}

Let $E$ be a commutative symmetric graded $k$-algebra (with $E_{\geq 1}$ 
not necessarily finitely generated) 
such that $E_0=k$. 
We have a pair of adjoint functors
   $$V:\Gr E\rightleftarrows\Gr^\Sigma E:U,$$
where $U$ is the forgetful functor. The functor $V$ is fully determined by its values on projective generators
mapping $E(-n)$ to $F_nE$, $n\in\bb N$. In detail, every $M\in\Gr E$ is the coend
   $$M=\int^nM_n\otimes E(-n).$$
By definition, $V(M):=\int^nM_n\otimes F_nE$.

For every $n\geq 0$ consider a morphism in $\Gr^\Sigma E$
   \begin{equation}\label{strelka}
    a_n:E[n]\wedge_E F_nE\to E,
   \end{equation}
which is adjoint to $\id:E[n]\to[F_nE,E]=E[n]$.
We say that an injective symmetric module $Q\in\Gr^\Sigma E$ is {\it stably fibrant\/} if 
$[a_n,Q]:Q=[E,Q]\to[E[n]\wedge_E F_nE,Q]$ is an isomorphism for every $n\geq 0$. Using Corollary~\ref{inj}
and the fact that $[F_mE,-]$ respects isomorphisms,
the $m$-th shift $Q[m]$ of $Q$ is stably fibrant.

\begin{lemma}\label{doroga}
Suppose $M\in\Gr^\Sigma E$ is such that $M_0=\cdots=M_{n-1}=0$ and $U(M)$
is generated by $m_i\in M_n$, $i\in I$, as a non-symmetric right $E$-module. 
Then $VU(M)\cong M[n]\wedge_E F_nE$.
\end{lemma}

\begin{proof}
It follows from our assumption on $M$ that there is an exact sequence in $\Gr E$
   $$\oplus_{j\in J}E(-n_j)\to\oplus_IE(-n)\to U(M)\to 0$$
with $n_j\geq n$ for all $j\in J$.
As $V$ is left adjoint, the sequence
   $$\oplus_{j\in J}F_{n_j}E\to\oplus_IF_nE\to VU(M)\to 0$$
is exact in $\Gr^\Sigma E$. 
On the other hand, there is an exact sequence in $\Gr^\Sigma E$
   $$\oplus_{j\in J}F_{n_j-n}E\to\oplus_IE\to M[n]\to 0.$$
Smashing it with $F_nE$, we get the desired isomorphism.
\end{proof}

\begin{remark}
The assumptions on the symmetric module $M$ from the previous lemma imply 
$M[n]$ is of the form $(M_n,M_nE_1,M_nE_2,\ldots)$, where $\Sigma_m$ acts on $M_nE_m$
by $\tau(m_nx)=m_n\tau(x)$, $x\in E_m$, $\tau\in\Sigma_m$. Modules of this form will also be called
{\it suspension $E$-modules}. They are analogs of suspension spectra of spaces in topology.
\end{remark}

By Lemma~\ref{exact} each projective generator $F_nE$ of $\Gr^\Sigma E$ is flat.
By~\cite[Lemma~3.1]{SlSt} the functor $[-,Q]:\Gr^\Sigma E\to\Gr^\Sigma E$ is exact for 
any injective $Q\in\Gr^\Sigma E$. We set
   $$\Tors^\Sigma E:=\{T\in\Gr^\Sigma E\mid[T,Q]=0\textrm{ for all stably fibrant $Q$}\}.$$
By construction, $\Tors^\Sigma E$ is the smallest sh-localizing subcategory of $\Gr^\Sigma E$ 
(i.e. closed under smashing with $F_mE$-s) containing 
$\kr a_n$ and $\coker a_n$, $n\geq 0$. As we have mentioned above, the $m$-th shift
$Q[m]$ of a stably fibrant module is stably fibrant.

We say that a localizing subcategory $\cc T\subset \Gr^\Sigma E$ is {\it tensor\/} if $T\wedge_E M\in\cc T$
for any $T\in\cc T$ and $M\in\Gr^\Sigma E$.

\begin{lemma}\label{tensor}
$\Tors^\Sigma E$ is a localizing tensor subcategory of $\Gr^\Sigma E$. If $E$ is finitely generated 
as a non-symmetric algebra, then
$\Tors^\Sigma E$ contains $V(\Tors E)$.
\end{lemma}

\begin{proof}
$\Tors^\Sigma E$ is plainly localizing. We can write every $M\in\Gr^\Sigma E$ as the coend
   $$M=\int^nM_n\otimes F_nE.$$
Therefore one has for any $T\in\Tors^\Sigma E$
   $$M\wedge_E T\cong\int^nM_n\otimes(F_nE\wedge_E T)$$
As $\Tors^\Sigma E$ is localizing, $M\wedge_E T\in\Tors^\Sigma E$ provided that each
$F_nE\wedge_E T\in\Tors^\Sigma E$. The latter is true due to
   $$[F_nE\wedge_E T,Q]\cong[T,Q[n]]=0$$
for any stably fibrant $Q$. We see that $\Tors^\Sigma E$ is tensor.

Since $\im a_1=E_+$, the symmetric module \(k=E/E_+\) is in $\Tors^\Sigma E$. Moreover, 
$V(k(-n))\cong F_nE\wedge_E k$ is in $\Tors^\Sigma E$ as $\Tors^\Sigma E$ is tensor. 
A non-symmetric module concentrated in one degree is a
quotient of a sum of copies of $k(-n)$, so its image under $V$ is
in $\Tors^\Sigma E$. A bounded non-symmetric nonnegative module has a finite filtration
by such degree layers. Right exactness of $V$, followed by exactness
of $q$, implies the assertion for bounded modules. Since every non-symmetric
torsion module is a directed union of bounded finitely generated
submodules by using~\eqref{contain} and $V$ preserves
colimits, the desired containment follows.
\end{proof}

\begin{proposition}\label{tensorcor}
Under the conditions of Lemma~\ref{tensor} for every
$\Tors^\Sigma E$-closed module $M\in\Gr^\Sigma E$, the module $U(M)$ is $\Tors E$-closed in $\Gr E$.
\end{proposition}

\begin{proof}
We use the facts that 
$M$ is the kernel of a map between stably fibrant objects and that the functor $U:\Gr^\Sigma E\to\Gr E$ is exact.
Thus the assertion reduces to showing that $U(Q)$ is $\Tors E$-closed for every stably fibrant $Q$.

Given $s,t\geq 0$, stable fibrancy and adjunction give
\begin{equation}\label{kukuku}
\begin{aligned}
Q_t
&\cong \operatorname{Hom}_{\operatorname{Gr}^{\Sigma}E}(E,Q[t])\\
&\xrightarrow{\;\sim\;}
\operatorname{Hom}_{\operatorname{Gr}^{\Sigma}E}
(E[s]\wedge_EF_sE,Q[t])
&\cong
\operatorname{Hom}_{\operatorname{Gr}^{\Sigma}E}
(E[s],Q[t+s]).
\end{aligned}
\end{equation}
Under these identifications, \(u\in Q_t\) corresponds to multiplication by \(u\):
\[
E_{s+n}\longrightarrow Q_{t+s+n},
\qquad a\longmapsto ua.
\]
In particular,
$uE_{\geq s}=0$ implies $u=0$. Indeed, the corresponding morphism on the right of~\eqref{kukuku} is then zero, and~\eqref{kukuku} is injective.

Let \(M,N\) be symmetric \(E\)-modules and \(U(M)\) is generated by
homogeneous elements \(m_i\) of degrees \(d_i\leq D\). We claim that every non-symmetric map \(f:U(M)\to U(N)\) induces a symmetric map 
$f[D]:M[D]\longrightarrow N[D].$ 
To check this, write an element of \(M_{D+t}\) as \(\sum_i m_i e_i\).
For \(\tau\in\Sigma_t\), the permutation \(1_D\times\tau\) fixes the
first \(d_i\) positions and acts on \(e_i\) by \(1_{D-d_i}\times\tau\).
The module equivariance axiom and \(E\)-linearity of \(f\) give \[
f\bigl((1_D\times\tau)(m_i e_i)\bigr)
=f(m_i)\bigl((1_{D-d_i}\times\tau)e_i\bigr)
=(1_D\times\tau)f(m_i e_i).
\] 
Thus the shifted map is equivariant, and hence our claim.

Fix \(r\geq1\), \(m\geq0\), and a non-symmetric map 
$f:E_{\geq r}\longrightarrow UQ[m]$.
Since \(E\) is finitely generated as an ordinary algebra over 
\(k=E_0\), the right ideal \(E_{\geq r}\) is finitely
generated in bounded degrees. Choose \(D\geq r\) at least as large as all the chosen generating
degrees. As above, \(f[D]\) becomes a symmetric morphism 
\[
f[D]:E[D]=(E_{\geq r})[D]\longrightarrow Q[m+D].
\]
Apply~\eqref{kukuku} with \(s=D\) and \(t=m\). The symmetric morphism \(f[D]\) corresponds to a unique \(u\in Q_m\). Thus
\[
f(a)=ua\qquad\text{for every }a\in E_{\geq D}.
\]
Define the non-symmetric graded \(E\)-linear map
\[
g:E_{\geq r}\longrightarrow UQ[m],
\qquad
g(a)=f(a)-ua.
\]
Then \(g\) vanishes on \(E_{\geq D}\). Therefore it factors uniquely as
\[
E_{\geq r}\twoheadrightarrow
B:=E_{\geq r}/E_{\geq D}
\xrightarrow{\;\bar g\;}UQ[m].
\]
But \(B\in\Tors E\), and therefore \(g=0\) due to Lemma~\ref{tensor}.
We have proved that
\[
f(a)=ua\qquad\text{for every }a\in E_{\geq r}.
\]
This proves surjectivity of
$$\Phi_{r,m}:Q_m\longrightarrow
\operatorname{Hom}_{\operatorname{Gr}E}(E_{\geq r},UQ[m]),
\qquad
u\longmapsto (a\longmapsto ua).$$
Suppose \(u,v\in Q_m\) induce the same map on \(E_{\geq r}\). Then
$(u-v)E_{\geq r}=0$.
As above, \(u=v\). Thus $\Phi_{r,m}$ is injective as well. Our statement now follows from Proposition~\ref{vesma}.
\end{proof}

\begin{defs}
The {\it category of symmetric quasi-coherent sheaves $\QGr^\Sigma E$} over a commutative graded
symmetric $k$-algebra $E$ is defined as the quotient Grothendieck category $\Gr^\Sigma E/\Tors^\Sigma E$.

Let $q$ denote the quotient functor $\Gr^\Sigma E\to\QGr^\Sigma E$. 
We shall identify $\QGr^\Sigma E$ with the full subcategory of
$\Tors^\Sigma E$-closed modules. A morphism $f:M\to N$ in $\Gr^\Sigma E$ is called
a {\it stable equivalence\/} if $q(f)$ is an isomorphism in $\QGr^\Sigma E$. Note that stable equivalences 
are closed under direct limits.
\end{defs}

\begin{lemma}\label{tensorlem}
Under the conditions of Lemma~\ref{tensor} the functor $V:\Gr E\to\Gr^\Sigma E$ preserves stable equivalences.
\end{lemma}

\begin{proof}
Suppose $f:M\to N$ is a stable equivalence in $\Gr E$. We have an exact sequence
   $$0\to T\to M\lra fN\to T'\to 0$$
with $T,T'\in\Tors E$. It remains to apply the functor $\uhom_E(-,U(Q))$, where $Q$ is stably fibrant, 
to two short exact sequences $T\hookrightarrow M\twoheadrightarrow\im f$ and
$\im f\hookrightarrow N\twoheadrightarrow T'$ as well as Proposition~\ref{tensorcor}.
\end{proof}

The shift functor $M\mapsto M[n]$, $n\geq 0$, defines a
shift functor on $\QGr^\Sigma E$ for which we use the same notation. Indeed,
$M$ is the kernel of a morphism between stably fibrant modules which are closed under shifts.
Finally, we shall write $\cc O^\Sigma(-n)=q(F_nE)$. Note that $\QGr^\Sigma E$ is a
Grothendieck category with the family of generators $\{\cc O^\Sigma(-n)\}_{n\in\bb N}$. We will
write $\cc O^\Sigma$ to denote $q(E).$

The tensor product in $\Gr^\Sigma E$ induces a tensor product in $\QGr^\Sigma E$,
denoted by $\boxtimes$. More precisely, one sets
   $$X\boxtimes Y:=q(X\wedge_E Y)$$
for any $X,Y\in\QGr^\Sigma E$.

\begin{lemma}\label{d2}
Given $X,Y\in\Gr^\Sigma E$ there is a natural isomorphism in $\QGr^\Sigma E$:
$q(X)\boxtimes q(Y)\cong q(X\wedge_E Y)$. 
Moreover, the quadruple $(\QGr^\Sigma E,\boxtimes,[-,-],\cc O^\Sigma)$
defines a closed symmetric mo\-noi\-dal category with monoidal
unit $\cc O^\Sigma$ and internal Hom-object the same 
with that in $\Gr^\Sigma E$.
\end{lemma}

\begin{proof}
The first claim is proven similarly to \cite[Lemma~4.2]{GP1}. The second claim
is now straightforward if we use the description of the internal Hom-object given in the preceding section.
\end{proof}

Recall from~\cite[p.~109]{DM} that an object $X$ of a symmetric monoidal category $(\cc V,\otimes,e)$ is
{\it invertible\/} if the endofunctor $X\otimes-:\cc V\to\cc V$ is an equivalence.
Equivalently, there is an object $X'\in\cc V$ such that $X\otimes X'\cong e$.
For every $n\geq 0$ we shall write $\cc O^\Sigma(n)$ to denote $q(E[n])$. The morphism $a_n:E[n]\wedge_EF_nE\to E$
induces an isomorphism $\alpha_n:\cc O^\Sigma(n)\boxtimes\cc O^\Sigma(-n)\to\cc O^\Sigma$ in $\QGr^\Sigma E$.

\begin{corollary}\label{invertible}
For every integer $n$, the object $\cc O^\Sigma(n)\in\QGr^\Sigma E$ is 
invertible with $\cc O^\Sigma(n)^{-1}=\cc O^\Sigma(-n)$. Moreover,
$\cc O^\Sigma(m)\boxtimes\cc O^\Sigma(n)\cong\cc O^\Sigma(m+n)$ for all $m,n\in\bb Z$.
\end{corollary}

The author is grateful to Peter Bonart for the following lemma.

\begin{lemma}\label{bonart}
Let $(\cc V,\otimes,[-,-])$ be a closed symmetric monoidal category and let $X$ be an invertible object 
in $\cc V$. Then $X$ is strongly dualisable and there is an isomorphism of functors $[X,-]\cong X^{-1}\otimes-$.
\end{lemma}

\begin{proof}
We have natural isomorphisms in $\cc V$
   $$X^\vee:=[X,e]\cong[X\otimes X^{-1},e\otimes X^{-1}]\cong[e,X^{-1}]\cong X^{-1}.$$
The functor $X\otimes-:\cc V\to\cc V$ is left adjoint to both $[X,-]$ and $X^{-1}\otimes-$.
By uniqueness for adjoint functors there is a natural isomorphism $[X,-]\cong X^{-1}\otimes-$.
It induces a natural isomorphism $X^{-1}\otimes Y\cong[X,Y]$, hence $X$ is strongly dualizable.
\end{proof}

For any $X\in\QGr^\Sigma E$ and any integer $n$ we write $X(n)$ for $X\boxtimes\cc O^\Sigma(n)$.

\begin{corollary}\label{bonartcor}
For any $n\in\bb Z$ the object $\cc O^\Sigma(n)$ is strongly dualizable in $\QGr^\Sigma E$ and 
there is a natural isomorphism 
   $X(n)\lra{\cong}[\cc O^\Sigma(-n),X].$
In particular, for any $n\geq 0$ the object $X(n)$ is naturally isomorphic to $X[n]$.
\end{corollary}

\begin{defs}
Let $\cc C$ be a closed symmetric monoidal Grothendieck category. An object $X\in\cc C$
is said to be {\it enriched finitely presented\/} if the functor $[X,-]=\uhom_{\cc C}(X,-)$ preserves
direct limits. We say that $\cc C$ is {\it enriched locally finitely presented\/} if $\cc C$ has a family
of enriched finitely presented generators.

If the monoidal unit $e$ of $\cc C$ is finitely presented,
one can show similarly to~\cite[Lemma~4.5]{GG} that if $X$ is enriched finitely presented, then
$X$ is finitely presented in the usual sense.
\end{defs} 

\begin{lemma}\label{elfpc}
Let $\cc C$ be a closed symmetric monoidal Grothendieck category with strongly dualizable generators
$\{g_i\}_{i\in I}$. Then $\cc C$ is enriched locally finitely presented. If the monoidal unit $e$ is finitely presented then
$\{g_i\}_{i\in I}$ is a family of finitely presented generators for $\cc C$.
\end{lemma}

\begin{proof}
As every generator $g_i$ is strongly dualizable, there is an isomorphism of functors
$[g_i,-]\cong g_i^\vee\otimes-$, where $g_i^\vee=[g_i,e]$ is dual to $g_i$. Our statement follows 
from the fact that $g_i^\vee\otimes-$ preserves direct limits.
\end{proof}

\begin{corollary}\label{bonartcor2}
The Grothendieck category $\QGr^\Sigma E$ is enriched locally finitely presented.
\end{corollary}

\begin{proof}
By Corollary~\ref{bonartcor} each generator $\cc O^\Sigma(-n)$, $n\geq 0$, of $\QGr^\Sigma E$ is
strongly dualizable. The statement now follows from Lemma~\ref{elfpc}.
\end{proof}

We document the results of this section as follows.

\begin{theorem}\label{itogo}
Let $E$ be a commutative graded symmetric $k$-algebra. Then $\QGr^\Sigma E$ is a closed symmetric
monoidal enriched locally finitely presented Grothendieck category with a family of invertible
generators $\{\cc O^\Sigma(n)\}_{n\in\bb Z}$ such that $\cc O^\Sigma(m)\boxtimes\cc O^\Sigma(n)\cong\cc O^\Sigma(m+n)$.
\end{theorem}

\section{Reconstruction Theorem}\label{reconstrsection}

Let $E$ be a commutative symmetric $k$-algebra which is finitely generated as a non-symmetric graded
$k$-algebra. The pair of adjoint functors
$V:\Gr E\leftrightarrows\Gr^\Sigma E:U$ can be extended to a pair of adjoint functors
      $$F:\QGr E\leftrightarrows\QGr^\Sigma E:U$$
Indeed, by Proposition~\ref{tensorcor} the restriction of the forgetful functor $U:\Gr^\Sigma E\to\Gr E$ 
to $\QGr^\Sigma E$ lands in $\QGr E$. Its right adjoint $F$ takes $X\in\QGr E$ to $q(V(X))$,
where $q:\Gr^\Sigma E\to\QGr^\Sigma E$ is the $\Tors^\Sigma$-localization functor.

We are now in a position to prove the main result of the paper. It is an analog for~\cite[Theorem~4.2.5]{HSS}
reconstructing $S^1$-spectra out of symmetric spectra.

\begin{theorem}[Reconstruction]\label{reconstr}
Suppose $R$ is a commutative finitely generated graded ring with $R_0$
being a $\bb Q$-algebra. Then the functors
   $$F:\QGr R\leftrightarrows\QGr^\Sigma R:U$$
are equivalences of categories and quasi-inverse to each other. Here $R$ is regarded
as a commutative graded symmetric $R_0$-algebra with trivial action of the symmetric groups.
\end{theorem}

\begin{proof}
It is enough to show that $U$ is an equivalence.
Suppose $Q$ is injective in $\Gr R$. Regarding $Q$ as a symmetric $R$-module with trivial action
(see Definition~\ref{emodule}),
we claim that it is injective in $\Gr^\Sigma R$. Indeed, let $i:M\hookrightarrow N$ be a monomorphism in
$\Gr^\Sigma R$, and let $f:M\to Q$ be a map of symmetric modules. Then $f$ factors through
$M/\Sigma:=(M_0,M_1/\Sigma_1,M_2/\Sigma_2,\ldots)$, where 
$M_n/\Sigma_n=M_n/\langle m-\sigma(m)\rangle_{m\in M_n,\sigma\in\Sigma_n}$. Note that
$M/\Sigma\in\Gr^\Sigma R$, where symmetric groups act trivially on $M/\Sigma$.
As $R$ is a $\bb Q$-algebra by assumption, the map $i$ is a degreewise split 
monomorphism of rational representations of symmetric groups
due to Maschke's Theorem. It follows that the induced map $i/\Sigma:M/\Sigma\to N/\Sigma$
is a degreewise monomorphism of graded $R$-modules. 
As every morphism $M/\Sigma\to Q$ is extended to $N/\Sigma$,
it follows that $f$ is extended to $N$, and hence $Q$ is an injective symmetric module as claimed.

Next, if $Q$ is torsionfree in $\Gr R$, then it is stably fibrant in $\Gr^\Sigma R$. Indeed, $R/\Sigma=R$
due to trivial action of the symmetric groups. Therefore
   $$[a_n,Q]:Q\to[R[n]\wedge_R F_nR,Q]$$
induced by the map~\eqref{strelka} is isomorphic to $Q\to\uhom_R(R_{\geq n},Q)$. By Proposition~\ref{vesma}
the latter arrow is an isomorphism. We see that $Q$ is stably fibrant.

Denote by
\[
\operatorname{tr}:\Gr R\longrightarrow \Gr^\Sigma R
\] 
the functor taking $X$ to the symmetric module which is $X$ itself with the trivial symmetric action.
This functor is fully faithful. Its left adjoint is the degreewise coinvariants functor $(-)/\Sigma$.
Since every $\Tors R$-closed module is the kernel of a map between torsionfree injectives, we see that the restriction of
$\operatorname{tr}$ to $\QGr R$ lands in $\QGr^\Sigma R$. Thus we obtain a fully faithful functor
   \[
    \operatorname{tr}:\QGr R\longrightarrow \QGr^\Sigma R.
   \] 

By construction, $U\circ\operatorname{tr}=\id_{\QGr R}.$ Therefore $U$ is an equivalence if and only if $\operatorname{tr}$ is.
Thus it remains to show that $\operatorname{tr}$ is essentially surjective. To this end, 
fix \(n\geq0\) and \(\sigma\in\Sigma_n\). Let
\[
\rho_\sigma:F_nR\longrightarrow F_nR
\]
be the endomorphism corresponding to \(\sigma\) acting on the $n$th degree generator. Under
\[
\operatorname{Hom}_{\Gr^\Sigma R}(F_nR,Y)\cong Y_n,
\]
precomposition with \(\rho_\sigma\) is nothing but the action of \(\sigma\) on \(Y_n\).

We claim that
\begin{equation}\label{nnn}
a_n\circ
\bigl(1_{R[n]}\wedge_R\rho_\sigma\bigr)
=a_n.
\end{equation}
Recall that
\[a_n:R[n]\wedge_RF_nR\longrightarrow R\]
is the evaluation morphism under the identification
$R[n]\cong[F_nR,R].$ First, consider the endomorphism \(\rho_\sigma\). 
Let
$g_n=1\in(F_nR)_n=R_0\Sigma_n$
be the generator. The endomorphism
\[
\rho_\sigma:F_nR\longrightarrow F_nR
\]is determined by
$\rho_\sigma(g_n)=\sigma g_n.$ Every symmetric module map \(f:F_nR\to N\) is determined by \(f(g_n)\in N_n\). Therefore
\begin{equation}\label{seven}
(f\circ\rho_\sigma)(g_n)
=f(\sigma g_n)
=\sigma f(g_n).
\end{equation}
Apply this to \(N=R[d]\). Its $n$th degree component is \(R_{d+n}\), with \(\Sigma_n\) acting through the subgroup
\[
\Sigma_n\longrightarrow\Sigma_{d+n},
\qquad \sigma\longmapsto1_d\times\sigma.
\]Since all symmetric group actions on \(R\) are trivial, \eqref{seven} gives
$f\circ\rho_\sigma=f$ for every \(f:F_nR\to R[d]\).
But
\[
[F_nR,R]_d
=
\operatorname{Hom}_{\operatorname{Gr}^{\Sigma}R}(F_nR,R[d]).
\]
Thus precomposition induces the identity in every degree: $[\rho_\sigma,R]
=\operatorname{id}_{[F_nR,R]}$.

In any closed symmetric monoidal category, an endomorphism \(v:A\to A\) satisfies
\[
\operatorname{ev}\circ([v,B]\otimes1_A)
=
\operatorname{ev}\circ(1_{[A,B]}\otimes v),
\]
with \(\operatorname{ev}:[A,B]\otimes A\to B\) the canonical morphism.
Formally, both sides of the latter equality correspond under tensor–Hom adjunction to the same morphism
\[
[v,B]:[A,B]\longrightarrow[A,B].
\]Taking \(A=F_nR\), \(B=R\), and \(v=\rho_\sigma\), one gets
\[
\operatorname{ev}
=
\operatorname{ev}\circ
(1_{[F_nR,R]}\wedge_R\rho_\sigma).
\]
Finally, identify \([F_nR,R]\) with \(R[n]\). Evaluation becomes \(a_n\), giving exactly \eqref{nnn}.

One has that \(q(a_n)\) is an isomorphism in $\QGr^\Sigma R$. Moreover,
\[
q(R[n])\boxtimes q(F_nR)
\cong q(R),
\]
so \(q(R[n])\) is tensor-invertible. Applying \(q\) to \eqref{nnn}, 
first cancelling \(q(a_n)\) and then cancelling tensor multiplication by \(q(R[n])\), gives
\begin{equation}\label{eight}
q(\rho_\sigma)=1_{q(F_nR)}.
\end{equation}
Let \(Y\in\QGr^\Sigma R\). Localization adjunction gives
\[
\operatorname{Hom}_{\Gr^\Sigma R}(F_nR,Y)
\cong
\operatorname{Hom}_{\QGr^\Sigma R}
\bigl(q(F_nR),Y\bigr).
\]
By \eqref{eight}, precomposition with \(\rho_\sigma\) is the identity. We see that \(\sigma\) acts trivially on \(Y_n\).
It follows that $Y=\textrm{tr}(UY)$ and $\textrm{tr}$ is essentially surjective.
This completes the proof of the theorem.
\end{proof}

A theorem of Serre~\cite{Se}, \cite[Proposition~30.14.4]{Stack} together 
with Theorems~\ref{equivqgr} and~\ref{reconstr} imply the following result.

\begin{corollary}
Let $X=\Proj(R)$ be the projective scheme associated with a commutative graded ring $R$ with $R_0$ being a Noetherian
$\bb Q$-algebra. Suppose $R$ is generated by finitely many elements of $R_1$.
Then the composite functor
   $$\Qcoh X\lra{\Gamma_*}\QGr^{\bb Z}R\lra{I}\QGr R\lra{F}\QGr^\Sigma R$$
is an equivalence of categories, where $\Gamma_*(\cc F)=\bigoplus_{d=-\infty}^{+\infty}H^0(X,\cc F\otimes\cc O(d))$.
\end{corollary}

The proof of Theorem~\ref{reconstr} uses Maschke's theorem, which ensures exactness of coinvariants when $R_0$ is a $\mathbb Q$-algebra. 
We expect an analogous result in positive characteristic. It will require a derived formulation involving suitable triangulated localizations.

\section{The topology of symmetric projective spaces}\label{sectionproj}

Throughout this section $E$ is a commutative symmetric graded $k$-algebra (with $E_{\geq 1}$ not necessarily finitely generated)
such that $E_0=k$. We 
introduce and study symmetric projective spaces associated with commutative symmetric graded $k$-algebras.
They extend the classical projective spaces. We recover classical projective schemes in the commutative case. We start with preparations.

We say that $M\in\Gr^\Sigma E$ is 
{\it finitely generated\/} if the functor $\Gr^\Sigma E(M,-)$ respects directed unions. 
If $M\in\Gr^\Sigma E$ and $N$ is a non-symmetric graded submodule of $M$ 
(i.e. $N$ is a subobject of $M$ in $\Gr E$), the
{\it $\Sigma$-closure of $N$ in $M$} is
   $$N^\Sigma:=\{\sigma_1(n_1)+\cdots+\sigma_k(n_k)\mid n_1,\ldots,n_k \textrm{ are homogeneous elements of $N$},
       \sigma_1\in\Sigma_{\deg n_1},\ldots,\sigma_k\in\Sigma_{\deg n_k}\}.$$
       
\begin{lemma}\label{fingen}
$N^\Sigma$ is a subobject of $M$ in $\Gr^\Sigma E$. If $N$ is finitely generated in $\Gr E$, then 
$N^\Sigma$ is finitely generated in $\Gr^\Sigma E$.
\end{lemma}

\begin{proof}
Clearly, $N^\Sigma$ is closed under addition. We have $N^\Sigma=\bigoplus_{t\in\bb N}N^\Sigma_t$, where
   $$N^\Sigma_t=\{\tau_1(n_1)+\cdots+\tau_s(n_s)\mid\tau_i\in\Sigma_t,n_i\in N_t\}.$$
By definition, $\Sigma_t$ acts on $N_t^\Sigma$ by the rule:
   $$\sigma(\tau_1(n_1)+\cdots+\tau_s(n_s)):=(\sigma\tau_1)(n_1)+\cdots+(\sigma\tau_s)(n_s),\quad\sigma\in\Sigma_t.$$
If $x\in E_\ell$ then
   $$(\tau_1(n_1)+\cdots+\tau_s(n_s))\cdot x:=(\tau_1\times 1)(n_1x)+\cdots+(\tau_s\times 1)(n_sx)\in N^\Sigma_{t+\ell}.$$
This induces a pairing $N^\Sigma_t\otimes E_{\ell}\to N^\Sigma_{t+\ell}$ which is clearly $\Sigma_t\times\Sigma_\ell$-equivariant.
Thus $N^\Sigma$ is a symmetric right $E$-module.

Suppose $N$ is finitely generated in $\Gr E$. Consider a map $f:N^\Sigma\to\sum_I L_i$ in $\Gr^\Sigma E$.
As the forgetful functor $U:\Gr^\Sigma E\to\Gr E$ respects all colimits, the restriction of $f$ to $N$ factors through
some $L_i$. It follows that $f$ factors through $L_i$, and hence $N^\Sigma$ is finitely generated.
\end{proof}

Suppose $N$ is a finitely generated symmetric submodule of $M$. Then there is an epimorphism
$\kappa:\bigoplus_{i=1}^u F_{\ell_i}E\twoheadrightarrow N$ corresponding to a morphism of non-symmetric graded modules
$\lambda:\bigoplus_{i=1}^u E(-\ell_i)\to N$. Denote the image of $\lambda$ by $L$.

\begin{corollary}\label{carbery}
$N$ is the $\Sigma$-closure of $L$.
\end{corollary}

\begin{remark}\label{lingren}
Recall that $F_nE=(0,\bl{n-1}\ldots,0,k\Sigma_n\otimes E_0,k\Sigma_{n+1}\otimes_{k\Sigma_1} E_1,k\Sigma_{n+2}\otimes_{k\Sigma_2} E_2,\ldots)$.
Then $E(-n)=(0,\bl{n-1}\ldots,0,E_0,E_1,E_2,\ldots)$ is a non-symmetric submodule of $F_nE$, where each $x\in E_{n+m}$ is regarded as the element
$(1,x)\in k\Sigma_{n+m}\otimes_{k\Sigma_m} E_m$. One has that $F_nE$ is the $\Sigma$-closure of $E(-n)$.
\end{remark}

A {\it $\Sigma$-ideal of $E$\/} is a subobject $I$ of $E$ in $\Gr^\Sigma E$. We say that a $\Sigma$-ideal $I$ is {\it principal\/}
if it is the $\Sigma$-closure of the (non-symmetric) right ideal $xE$ for some homogeneous $x\in E$. In this case
we write $(x)^\Sigma$ for $I$.
More generally, note that if a $\Sigma$-ideal is finitely generated, there are $x_1,\ldots,x_n\in E$
such that $I$ is the $\Sigma$-closure of the non-symmetric finitely generated right ideal $x_1E+\cdots+x_nE$. In this case we
write $I=(x_1,\ldots,x_n)^\Sigma$.

It is useful to have the following concept.

\begin{defs}
We say that $E$ is {\it finitely $\Sigma$-generated by $x_1,\ldots,x_n\in E$} if every element of $E$ is a finite $k$-linear
combination of permuted monomials $\sigma(x_1^{a_1}x_2^{a_2}\cdots x_n^{a_n})$, where $\sigma$
is a permutation in $\Sigma_{\deg(x_1)a_1+\cdots+\deg(x_n)a_n}$.
Note that $E$ is finitely $\Sigma$-generated if and only if the $\Sigma$-ideal $E_{\geq 1}$ is finitely generated.
\end{defs}

\begin{lemma}\label{ideals}
Any $\Sigma$-ideal is a two-sided ideal after forgetting the symmetric structure on $E$. 
\end{lemma}

\begin{proof}
By definition, any $\Sigma$-ideal $I$ of $E$ is a right (non-symmetric) ideal of $E$.
Given two homogeneous elements $x\in I$ and $y\in E$ of degrees $m$ and $n$, one has $xy\in I$
and $yx=\chi_{m,n}(xy)\in I$.
\end{proof}

Given two $\Sigma$-ideals $I,J$ of $E$, their product $IJ$ is defined as the image of the map
   $$I\wedge J\to E\wedge E\to E.$$
Clearly, $IJ$ is a $\Sigma$-ideal. It is the $\Sigma$-closure of the naive product of (non-symmetric)
ideals $I$ and $J$. 

By a {\it prime $\Sigma$-ideal\/} we mean a proper $\Sigma$-ideal $P$ of $E$ satisfying the property $x\in P$
or $y\in P$ whenever $xy\in P$ for $x,y\in E$. 

\begin{lemma}\label{prime}
The following conditions are equivalent for a $\Sigma$-ideal $P$:

\begin{enumerate}
\item $P$ is prime;
\item for any two homogeneous elements
$x,y\in E$ the condition $xy\in P$ implies that $x\in P$ or $y\in P$;
\item for any two $\Sigma$-ideals
$I,J\subset E$ the condition $IJ\subset P$ implies that $I\subset P$ or $J\subset P$.
\end{enumerate}
\end{lemma}

\begin{proof}
$(1)\Leftrightarrow (2)$. The proof literally repeats that of a sublemma in~\cite[Section~5]{GP1}.

$(3)\Rightarrow (2)$. This immediately follows if we consider principal $\Sigma$-ideals 
for homogeneous $x,y\in E$ and note that $(x)^\Sigma(y)^\Sigma=(xy)^\Sigma$.

$(2)\Rightarrow (3)$. Suppose $I,J$ are $\Sigma$-ideals of $E$ such that $IJ\subset P$
but $I\nsubseteq P$ and $J\nsubseteq P$. Then there are homogeneous $x\in I$ and $y\in J$
such that $x\notin P$ and $y\notin P$. By assumption, $xy$ is not in $P$. On the other hand,
$xy\in IJ\subset P$, a contradiction.
\end{proof}

\begin{defs}\label{projsigma}
The {\it symmetric projective space $\Proj^\Sigma E$\/} is 
a topological space whose points are the  prime $\Sigma$-ideals not containing $E_{\geq 1}$. 
Zariski's topology of $\Proj^\Sigma E$ is defined by taking the closed sets to be the
sets of the form $V(I)=\{P\in\Proj^\Sigma E\mid P\supseteq I\}$ for some
$\Sigma$-ideal $I$ of $E$. We set $D(I):=\Proj^\Sigma E\setminus V(I)$.
\end{defs}

We will endow this space with a sheaf of {\it commutative\/} rings $\cc O_{\Proj^\Sigma E}$ such that the resulting pair 
$(\Proj^\Sigma E,\cc O_{\Proj^\Sigma E})$ will be a ringed space. By the preceding lemma one has
   $$V(IJ)=V(I)\cup V(J),\quad V(\Sigma_{a\in A}I_a)=\cap_{a\in A}V(I_a).$$
Note that 
   $D(a_0):=\{P\in\Proj^\Sigma E\mid a_0\notin P\}=\bigcup_{b\in E_{\geq 1}}D(a_0b)$ for any $a_0\in E_0$. It follows that
$\Proj^\Sigma E$ has a basis of open sets $D(f)$ with $f\in E_n$, $n\geq 1$, which we call {\it standard opens}.
The intersection of two standard opens is another: $D(f)\cap D(g)=D(fg)$, for homogeneous elements $f,g\in E$ of positive degree.

\begin{exs}\label{primerproj}
(1) If $E$ is graded commutative, then $\Proj^\Sigma E=\Proj E$ by construction. 

(2) Given a free $k$-module $V$, we define the {\it symmetric projective space of $V$\/} by
   $$\Proj^\Sigma(V):=\Proj^\Sigma(T(V)),$$
where $T(V)$ is the tensor algebra for $V$. The classical projective space $\Proj(V)$ of $V$ is defined
as $\Proj(S(V))$, where $S(V)$ is the commutative graded $k$-algebra defined in each homogeneous degree $n$
by $V^{\otimes n}/(v_1\otimes\cdots\otimes v_n-v_{\sigma(1)}\otimes\cdots\otimes v_{\sigma(n)})$, $\sigma\in\Sigma_n$.
It follows that $\Proj(S(V))$ is homeomorphic to the closed subset $V(I)$ of $\Proj^\Sigma(V)$
endowed with the subspace topology, where $I$ is the $\Sigma$-ideal of $T(V)$ generated 
by the homogeneous elements $v_1\otimes\cdots\otimes v_n-v_{\sigma(1)}\otimes\cdots\otimes v_{\sigma(n)}$, $\sigma\in\Sigma_n$.
\end{exs}

\begin{lemma}\label{compact}
An open subset of $\Proj^\Sigma E$ is quasi-compact if and only if it is of the form 
$D(I)$ for some finitely generated $\Sigma$-ideal $I=(a_{1},\ldots,a_{n})^\Sigma$
with $a_1,\ldots,a_n\in E_{\geq1}$.
\end{lemma}

\begin{proof}
Let $U\subset \Proj^\Sigma E$ be open quasi-compact. Then 
$U=\bigcup_{s\in S}D(a_s)$ with each $a_s\in E_{\geq 1}$. There is a finite subset $T=\{s_1,\ldots,s_n\}\subset S$ such that
$U=\bigcup_{t\in T}D(a_t)$. Set $I=(a_{s_1},\ldots,a_{s_n})^\Sigma$. Then $I$ is a finitely generated
$\Sigma$-ideal and $U=D(I)$.

Conversely, if $I=(a_{1},\ldots,a_{n})^\Sigma$ is a finitely generated $\Sigma$-ideal with $a_1,\ldots,a_n\in E_{\geq1}$, then
$D(I)=D(a_1)\cup\cdots\cup D(a_n)$. So it is enough to verify that $D(a)$ is quasi-compact for any $a\in E_{\geq 1}$.
Consider a cover $D(a)=\bigcup_\Lambda D(I_\lambda)$ of $D(a)$ by open sets
$D(I_\lambda)$. Assume $D(a)\ne D(I_{\lambda_1})\cup\cdots\cup
D(I_{\lambda_n})$ for any $\lambda_1,\ldots,\lambda_n\in\Lambda$.
Set $I:=\sum_\Lambda I_\lambda$. Then $a\notin I$ because otherwise
$a\in I_{\lambda_1}+\cdots+I_{\lambda_n}$ for some
$\lambda_1,\ldots,\lambda_n\in\Lambda$ and then
$D(a)=D(I_{\lambda_1})\cup\cdots\cup D(I_{\lambda_n})$. It also
follows that $a^t\notin I$ for any $t$.

\begin{sublem}
Let $Q$ be a $\Sigma$-ideal with $a^t\notin Q$ for all $t$, $Q\supseteq I$,
and let $Q$ be maximal such. Then $Q$ is prime.
\end{sublem}

\begin{proof}
By the previous lemma it is enough to check that for any two
homogeneous elements $b,c\in E$ the condition $bc\in Q$ implies
$b\in Q$ or $c\in Q$. Let $b,c\in E$ be such that $bc\in Q$ and
$b,c\notin Q$. Then $a^t=q+\sum_{i=1}^n\sigma_i(br_i)\in Q+(b)^\Sigma$ and $a^s=q'+\sum_{j=1}^m\tau_j(cr'_j)\in Q+(c)^\Sigma$ for
some $s,t\in\bb N$, $q,q'\in Q$ and permutations $\sigma_i$-s, $\tau_j$-s. One has,
   $$\sigma_i(br_i)\tau_j(cr'_j)=(\sigma_i\times\tau_j)(br_icr'_j)=(\sigma_i\times\tau_j)(\id\times\chi_{\deg(r_i),\deg(c)}\times\id)(bcr_ir'_j)\in Q.$$
We see that
$a^{t+s}=q''+\sum_{i,j}\sigma_i(br_i)\tau_j(cr'_j)\in Q$, a contradiction.
\end{proof}

Choose $Q$ as in the sublemma.  Note that $Q\nsupseteq E_{\geq 1}$ so $Q\in
\Proj E$. Since $a\notin Q$, $Q\in D(a)$ and so $Q\in
D(I_{\lambda_0})$ for some $\lambda_0\in\Lambda$. It follows that
$Q\nsupseteq I_{\lambda_0}$, a contradiction. So $D(a)$ is
quasi-compact. 
\end{proof}

\begin{corollary}
The space $\Proj^\Sigma E$ is quasi-compact whenever the $\Sigma$-ideal $E_{\geq 1}$
is finitely generated.
\end{corollary}

\begin{lemma}\label{t0}
The space $\Proj^\Sigma E$ is $T_0$.
\end{lemma}

\begin{proof}
Let $P,Q$ be two different prime $\Sigma$-ideals of $E$. Without loss
of generality we may assume that there is a homogeneous element
$a\in P$ such that $a\notin Q$. It follows that $P\notin D(a)$ but
$Q\in D(a)$. Therefore $\Proj^\Sigma E$ is $T_0$.
\end{proof}

\begin{lemma}\label{irreducible}
Every non-empty irreducible closed subset
$Z$ of $\Proj^\Sigma E$ has a generic point.
\end{lemma}

\begin{proof}
Let \(Z\) be a nonempty irreducible
closed subset and put \[
P_Z=\bigcap_{P\in Z}P.
\] It is a symmetric ideal. Every \(P_0\in Z\) contains
\(P_Z\), and so \(P_Z\) is proper and does not contain \(E_+\).
Let $a,b$ be homogeneous such that \(ab\in P_Z\), then \[
Z=(Z\cap V(a))\cup(Z\cap V(b)).
\] Irreducibility of $Z$ implies \(Z\subseteq V(a)\) or \(Z\subseteq V(b)\),
hence \(a\in P_Z\) or \(b\in P_Z\). By Lemma~\ref{prime} \(P_Z\) is prime.
Finally, if \(Z=V(I)\), then \(I\subseteq P_Z\) and every point of \(Z\)
contains \(P_Z\). Therefore \(V(P_Z)=Z\), so \(P_Z\) is a generic point.
\end{proof}

Recall from~\cite{Hoc} that a topological space is {\it
spectral\/} if it is $T_0$ and quasi-compact, the quasi-compact open
subsets are closed under finite intersections and form an open
basis, and every non-empty irreducible closed subset has a generic point.

We document the above statements as follows. 

\begin{theorem}\label{d3}
Suppose $E$ is finitely $\Sigma$-generated. Then the space $\Proj^\Sigma E$ is spectral.
\end{theorem}

Next, we want define a structure sheaf of {\it commutative\/} rings $\cc O_{\Proj^\Sigma E}$.
Denote by $\cc B$ the basis of quasi-compact open subsets of $\Proj^\Sigma E$. If $D\in\cc B$ we set
   $$\cc F_{ D}:=\{I\subset E\mid V(I)\subset \Proj^\Sigma E\setminus D\}.$$
By construction, $\cc F_{ D}\subset \cc F_{ D'}$ whenever $D'\subset D$.

For every $n\geq 0$ and $I\in\cc F_D$ consider a morphism in $\Gr^\Sigma E$
   \begin{equation*}
    d_{n,I}:I[n]\wedge_E F_nE\to E,
   \end{equation*}
which is adjoint to the inclusion $\id:I[n]\to[F_nE,E]=E[n]$. 
We say that an injective symmetric module $Q\in\Gr^\Sigma E$ is {\it $\cc F_D$-stably fibrant\/} if 
$[d_{n,I},Q]:Q=[E,Q]\to[I[n]\wedge_E F_nE,Q]$ is an isomorphism for every $n\geq 0$ and $I\in\cc F_D$. Using Corollary~\ref{inj}
and the fact that $[F_mE,-]$ respects isomorphisms,
the $m$-th shift $Q[m]$ of $Q$ is $\cc F_D$-stably fibrant. Note that $Q$ is stably fibrant as well, because $E\in\cc F_D$.

Let $\cc S_D$ be the localizing subcategory of $\Gr^\Sigma E$ cogenerated by the $\cc F_D$-stably fibrant objects.
Note that $\Tors^\Sigma E\subset\cc S_D$. Similarly to the proof of Lemma~\ref{tensor}, $\cc S_D$ is
a tensor subcategory. Let $q_D:\Gr^\Sigma E\to\Gr^\Sigma E/\cc S_D$ be the quotient functor. 
The tensor product in $\Gr^\Sigma E$ induces a tensor product in $\Gr^\Sigma E/\cc S_D$,
denoted by $\boxtimes_D$. More precisely, one sets
   $$X\boxtimes_D Y:=q_D(X\wedge_E Y)$$
for any $X,Y\in\Gr^\Sigma E/\cc S_D$. Similarly to Lemma~\ref{d2}
the quadruple $(\Gr^\Sigma E/\cc S_D,\boxtimes_D,[-,-],\cc O^D)$, $\cc O^D:=q_D(E)$,
defines a closed symmetric mo\-noi\-dal category with monoidal
unit $\cc O^D$ and internal Hom-object the same 
with that in $\Gr^\Sigma E$.

As $\cc S_D\supset\Tors^\Sigma E$, the category $\Gr^\Sigma E/\cc S_D$ equals the quotient
category $\QGr^\Sigma E/\cc T_D$, where 
   $$\cc T_D=\{X\in\QGr^\Sigma E\mid[X,Q]=0\textrm{\ for all $\cc F_D$-stably fibrant $Q$}\}.$$
We define a structure sheaf on $\cc O_{\Proj^\Sigma E}$ as follows. We
obtain a presheaf of rings on the basis $\cc B$ by
   $$D\mapsto\End_{\QGr^\Sigma E/\cc T_D}(\cc O^D).$$ 
If $D'\subseteq D$ are open subsets in $\cc B$, then $\cc T_D\subset\cc T_{D'}$ and the restriction map
   $$\End_{\QGr^\Sigma E/\cc T_D}(\cc O^D)\to\End_{\QGr^\Sigma E/\cc T_{D'}}(\cc O^{D'})$$
is induced by the quotient functor $\QGr^\Sigma E/\cc T_D\to\QGr^\Sigma E/\cc T_{D'}$. 
This is a presheaf of commutative rings on $\cc B$ by~\cite[Proposition~XI.2.4]{K}. 
It is extended to a presheaf on $\Proj^\Sigma E$
as follows. For any open $U$ in $\Proj^\Sigma E$ let
   $$U\longmapsto\lo{}_{D\subset U,\, D\in\cc B}\End_{\QGr^\Sigma E/\cc T_D}(\cc O^D).$$
Its sheafification is called the {\it structure sheaf\/} of
$\Proj^\Sigma E$ and is denoted by $\cc O_{\Proj^\Sigma E}$. By construction,
this is a sheaf of commutative rings on $\Proj^\Sigma E$.

\begin{defs}\label{urok}
The {\it symmetric projective scheme of $E$\/} is the ringed space $(\Proj^\Sigma E,\cc O_{\Proj^\Sigma E})$.
\end{defs}

The following theorem recovers classical projective schemes out of their symmetric counterparts.

\begin{theorem}\label{ramki}
Suppose $R$ is a commutative finitely generated graded ring with $R_0$
being a $\bb Q$-algebra. Then the projective scheme $(\Proj R,\cc O_{\Proj R})$
can be identified with $(\Proj^\Sigma R,\cc O_{\Proj^\Sigma R})$. Here $R$ is regarded
as a commutative graded symmetric $R_0$-algebra with trivial action of the symmetric groups.
\end{theorem}

\begin{proof}
Since the symmetric-group actions on \(R\) are trivial, its \(\Sigma\)-ideals are precisely its homogeneous ideals. 
Thus the prime ideals and the closed subsets agree in the two constructions, and hence
\[
X=\operatorname{Proj}^{\Sigma}R=\operatorname{Proj}R
\]as topological spaces.
Write
\[
q^\Sigma:\operatorname{Gr}^{\Sigma}R\longrightarrow\operatorname{QGr}^{\Sigma}R,
\qquad
q:\operatorname{Gr}R\longrightarrow\operatorname{QGr}R
\]for the quotient functors. By the proof of Theorem~\ref{reconstr}, every object of 
\(\operatorname{QGr}^{\Sigma}R\) has 
trivial symmetric-group actions, and the forgetful functor induces an equivalence of categories
\[
U:\operatorname{QGr}^{\Sigma}R\xrightarrow{\sim}\operatorname{QGr}R.
\]
We first identify the morphisms defining the localizations constructed in this section under this equivalence. 
Let \(I\subseteq R\) be a homogeneous ideal and let \(n,m\ge0\). For every $\Tors^\Sigma R$-closed symmetric module \(Y\), 
adjunction and reindexing give natural isomorphisms
\[
\begin{aligned}
&\operatorname{Hom}_{\operatorname{QGr}^{\Sigma}R}
\bigl(q^\Sigma(I[n]\wedge_R F_nR\wedge_R F_mR),Y\bigr)\\
&\quad\cong
\operatorname{Hom}_{\operatorname{Gr}^{\Sigma}R}
\bigl(I[n],Y[n+m]\bigr)\\
&\quad=
\operatorname{Hom}_{\operatorname{Gr}R}
\bigl(I[n],(UY)[n+m]\bigr)\\
&\quad\cong
\operatorname{Hom}_{\operatorname{Gr}R}
\bigl(I_{\ge n}(-m),UY\bigr)\\
&\quad\cong
\operatorname{Hom}_{\operatorname{QGr}R}
\bigl(q(I_{\ge n}(-m)),UY\bigr).
\end{aligned}
\]
The equality uses the trivial actions on both modules. The last isomorphism uses the
$\Tors R$-closedness of \(UY\), established in Theorem~\ref{reconstr}. Similarly,
\[
\operatorname{Hom}_{\operatorname{QGr}^{\Sigma}R}
(q^\Sigma F_mR,Y)
\cong
\operatorname{Hom}_{\operatorname{QGr}R}
(q(R(-m)),UY).
\]
It follows that precomposition with
\[
d_{n,I}\wedge_R F_mR:
I[n]\wedge_R F_nR\wedge_R F_mR\longrightarrow F_mR
\]
corresponds to restriction along
\[
I_{\ge n}(-m)\hookrightarrow R(-m).
\]
Indeed, the morphism determined by \(u\in(UY)_m\) is sent to \(a\mapsto ua\). Consequently,
\begin{equation}\label{anton}
U\bigl(q^\Sigma(d_{n,I}\wedge_R F_mR)\bigr)
\cong
q\bigl(I_{\ge n}(-m)\hookrightarrow R(-m)\bigr).
\end{equation}
One also has \(U(q^\Sigma R)\cong qR\) due to a natural isomorphism that uses full faithfulness of \(U\):
   $$\operatorname{Hom}_{\operatorname{QGr}R}
\bigl(U(q^\Sigma R),UY\bigr)
\cong\operatorname{Hom}_{\operatorname{QGr}R}(qR,UY).$$
Fix \(f\in R_d\), \(d>0\), and put \(D=D(f)\). We have
\begin{equation}\label{power}
I\in\mathcal F_D
\quad\Longleftrightarrow\quad
f^a\in I\text{ for some }a\ge1.
\end{equation}
One implication follows from primeness. For the converse, if no power of \(f\) belongs to \(I\), 
homogeneous prime-ideal existence gives a homogeneous prime containing \(I\) and avoiding \(f\) 
(this follows from Zorn's lemma and the proof of the sublemma above). Since \(f\) 
has positive degree, this prime belongs to \(V(I)\cap D\), contradicting \(I\in\mathcal F_D\).

Let \(\mathcal T_f\subseteq\operatorname{Gr}R\) denote the localizing subcategory of 
graded modules whose elements are annihilated by powers of \(f\). By the defining 
injective orthogonality characterization of \(\mathcal S_D\), its image in 
\(\operatorname{QGr}^{\Sigma}R\) is the localizing subcategory generated by the kernels and cokernels of
\[
q^\Sigma(d_{n,I}\wedge_R F_mR),
\qquad I\in\mathcal F_D,\quad n,m\ge0.
\]
By~\eqref{anton}, the corresponding kernels in \(\operatorname{QGr}R\) vanish, and the cokernels are
$q\bigl((R/I_{\ge n})(-m)\bigr)$.
Every such cokernel is \(f\)-power torsion. 
Indeed, by~\eqref{power}, \(f^a\in I\) for some \(a\). Choosing \(b\ge a\) sufficiently large that \(bd\ge n\), we obtain \(f^b\in I_{\ge n}\), so \(f^b(R/I_{\ge n})=0\).
Also, taking \(I=(f^a)\) and \(n=0\) produces all modules
\[
(R/(f^a))(-m),\qquad a\ge1,\quad m\ge0.
\]
These generate \(\mathcal T_f\). Namely, every homogeneous element \(x\in M_m\) annihilated by \(f^a\) is the image of \(1\) under a graded morphism
\[
(R/(f^a))(-m)\longrightarrow M,
\]and the direct sum of these morphisms is an epimorphism whenever \(M\in\mathcal T_f\).
Since \(\operatorname{Tors}R\subseteq\mathcal T_f\), transitivity of Gabriel localization now yields
\[
\operatorname{Gr}^{\Sigma}R/\mathcal S_D
\simeq
\operatorname{Gr}R/\mathcal T_f
\simeq
\operatorname{Gr}^{\mathbb Z}(R[f^{-1}]).
\]
The final equivalence is induced by graded localization. Its fully faithful right adjoint is restriction 
followed by nonnegative truncation. Under these equivalences, 
the localized unit \(q_DR\) corresponds to the graded module \(R[f^{-1}]\). Hence one gets a canonical ring isomorphism
\[
\operatorname{End}_{\operatorname{Gr}^{\Sigma}R/\mathcal S_D}(q_DR)
\cong
\operatorname{End}_{\operatorname{Gr}^{\mathbb Z}(R[f^{-1}])}
(R[f^{-1}])
\cong
(R[f^{-1}])_0.
\]
These isomorphisms are compatible with restrictions. If \(D(g)\subseteq D(f)\), then 
some power of \(g\) belongs to \((f)\), so \(f\) is invertible in \(R[g^{-1}]\). The further quotient functor 
corresponds to further graded localization, and its map on the rings is therefore the canonical map
\[
(R[f^{-1}])_0\longrightarrow(R[g^{-1}])_0.
\]Thus the presheaf defining \(\mathcal O_{\operatorname{Proj}^{\Sigma}R}\) agrees on the standard-open basis, including its restriction maps, with
\[
D(f)\longmapsto(R[f^{-1}])_0.
\]This is the standard-open description of the classical structure sheaf. Sheafification therefore gives a canonical isomorphism
\[
\mathcal O_{\operatorname{Proj}^{\Sigma}R}
\cong
\mathcal O_{\operatorname{Proj}R}.
\]
This completes the proof of the theorem.
\end{proof}

\appendix\section{Basic facts on Grothendieck categories}\label{poss}

In this section we recall basic facts about Grothendieck categories and Serre's localization
for the convenience of the reader. 

\begin{defs} {\rm
Recall that an object $A$ of a Grothendieck category 
$\cc C$ is {\em finitely generated\/} if
whenever there are subobjects $A_i\subseteq A$ for $i\in I$
satisfying $A=\sum_{i\in I}A_i$, then there is a finite subset
$J\subset I$ such that $A=\sum_{i\in J}A_i$. The category of
finitely generated subobjects of $\cc C$ is denoted by $\fg\cc C$.
The category is {\em locally finitely generated\/} provided that
every object $X\in\cc C$ is a directed sum $X=\sum_{i\in I}X_i$ of
finitely generated subobjects $X_i$, or equivalently, $\cc C$
possesses a family of finitely generated generators.}
\end{defs}

\begin{defs} {\rm
A finitely generated object $B\in\cc C$ is {\em finitely
presented\/} provided that every epimorphism $\eta:A\to B$ with $A$
finitely generated has a finitely generated kernel $\kr\eta$. The
subcategory of finitely presented objects of $\cc C$ is denoted by
$\fp\cc C$. Note that the
subcategory $\fp\cc C$ of $\cc C$ is closed under extensions.
Moreover, if
   $$0\to A\to B\to C\to 0$$
is a short exact sequence in $\cc C$ with $B$ finitely presented,
then $C$ is finitely presented \ifff $A$ is finitely generated. The
category is {\em locally finitely presented\/} provided that it is
locally finitely generated and every object $X\in\cc C$ is a direct
limit $X=\lp_{i\in I}X_i$ of finitely presented objects $X_i$, or
equivalently, $\cc C$ possesses a family of finitely presented
generators.
}
\end{defs}

A subcategory $\cc S$ of a Grothendieck category $\cc C$ is said to
be {\em Serre\/} if for any short exact sequence
   $$0\to{X'}\to X\to {X''}\to 0$$
$X',X''\in\cc S$ \ifff $X\in\cc S$. A Serre subcategory $\cc S$ of
$\cc C$ is said to be a {\em torsion class\/} or {\it localizing\/} if $\cc S$ is closed
under taking coproducts. An object $C$ of $\cc C$ is said to be {\it
torsionfree\/} if $(\cc S, C)=0$.  The pair consisting of a torsion
class and the corresponding class of torsionfree objects is referred
to as a {\em torsion theory}.   Given a torsion class $\cc S$ in
$\cc C$ the {\it quotient category $\cc C/\cc S$\/} is the full
subcategory with objects those torsionfree objects $C\in\cc C$
satisfying ${\rm Ext}^1(T,C)=0$ for every $T\in \cc S$.  The
inclusion functor $i:\cc S\to\cc C$ admits the right adjoint $t:\cc
C\to\cc S$ which takes every object $X\in\cc C$ to the maximal
subobject $t(X)$ of $X$ belonging to $\cc S$. The functor $t$ we
call the {\it torsion functor}.  Moreover, the inclusion functor
$i:\cc C/\cc S\to\cc C$ has a left adjoint, the {\it localization
functor\/} ${(-)}_{\cc S}:\cc C\to\cc C/\cc S$, which is also exact.
Then,
   $$\Hom_{\cc C}(X,Y)\cong\Hom_{\cc C/\cc S}(X_{\cc S},Y)$$
for all $X\in\cc C$ and $Y\in\cc C/\cc S$. A torsion class $\cc S$
is {\it of finite type\/} if the functor $i:\cc C/\cc S\to\cc C$
preserves directed sums. 

Let $\cc C$ be a Grothendieck category having a family of finitely
generated projective generators $\cc A=\{P_i\}_{i\in I}$. Let $\ff
F=\bigcup_{i\in I}\ff F^i$ be a family of subobjects, where each
$\ff F^i$ is a family of subobjects of $P_i$. We refer to $\ff F$ as
a {\em Gabriel filter\/} if the following axioms are satisfied:

\begin{itemize}
\item[$T1.$] $P_i\in\ff F^i$ for every $i\in I$;
\item[$T2.$] if $\ff a\in\ff F^i$ and $\mu:{P_j}\to{P_i}$
    then $\{\ff a:\mu\}=\mu^{-1}(\ff a)\in\ff F^j$;
\item[$T3.$] if $\ff a$ and $\ff b$ are subobjects of $P_i$ such that
    $\ff a\in\ff F^i$ and $\{\ff b:\mu\}\in\ff F^j$
    for any $\mu:{P_j}\to{P_i}$ with $\im\mu\subset\ff a$
    then $\ff b\in\ff F^i$.
\end{itemize}

\noindent In particular each $\ff F^i$ is a filter of subobjects of
$P_i$. A Gabriel filter is {\it of finite type} if each of these
filters has a  cofinal set of finitely generated objects (that is,
if for each $i$ and each $\ff a \in \ff F_i$ there is a finitely
generated $\ff b \in \ff F_i$ with $\ff a \supseteq \ff b$).

If $\cc A=\{A\}$ is a ring and $\ff a$ is a right ideal of $A$, then
for every endomorphism $\mu:A\to A$
   $$\mu^{-1}(\ff a)=\{\ff a:\mu(1)\}=\{a\in A\mid \mu(1)a\in\ff a\}.$$
On the other hand, if $x\in A$, then $\{\ff a:x\}=\mu^{-1}(\ff a)$,
where $\mu\in\End A$ is such that $\mu(1)=x$.

It is well-known (see, e.g., \cite{G}) that the map
   $$\cc S\longmapsto\ff F(\cc S)=\{\ff a\subseteq P_i\mid i\in I,\, P_i/\ff a\in\cc S\}$$
establishes a bijection between the Gabriel filters (respectively
Gabriel filters of finite type) and the torsion classes on $\cc C$
(respectively torsion classes of finite type on $\cc C$).

\section{Further properties of symmetric projective schemes}\label{appb}

We finish the paper with several useful remarks illustrating that 
the symmetric theory is an extension of the classical one and not an analogue of it. In particular, the
invariants discussed below yield new information: they
degenerate to classical ones precisely in the commutative situation.

\subsection*{The Serre twist becomes an invertible sheaf}
In noncommutative projective geometry the twist is an autoequivalence of
$\QGr A$, but it is not tensoring with an object, since $\QGr A$ carries no
tensor product. Passing to symmetric modules repairs this: each
$\cc O^{\Sigma}(n)$ is invertible, the twists compose as
$\cc O^{\Sigma}(m)\boxtimes\cc O^{\Sigma}(n)\cong
\cc O^{\Sigma}(m+n)$, so that there is a canonical homomorphism
\[
 \Z\longrightarrow\Pic\bigl(\QGr^{\Sigma}E\bigr),\qquad
 n\longmapsto\cc O^{\Sigma}(n),
\]
and every generator is strongly dualizable (Corollary~\ref{bonartcor}), whence
$\QGr^{\Sigma}E$ is enriched locally finitely presented
(Corollary~\ref{bonartcor2}). The symmetric structure is exactly what turns the twist
into an invertible sheaf, and the invertibility is not an extra hypothesis
but a consequence of the definitions.

\subsection*{A canonical parity}
Invertibility of $L=\cc O^{\Sigma}(-1)=q(F_1E)$ has a consequence
which seems to have no classical counterpart. The monoidal symmetry
$L\boxtimes L\to L\boxtimes L$ corresponds, under
$\End(L\boxtimes L)\cong\End(\cc O^{\Sigma})$, to an element
$\varepsilon$ of the commutative ring $\End(\cc O^{\Sigma})$ with
$\varepsilon^{2}=1$. The action of $\sigma\in\Sigma_n$ on
$q(F_nE)\cong L^{\boxtimes n}$ is $\varepsilon$ raised to the parity of
$\sigma$. If $2$ is invertible in $k$, the idempotents
$\tfrac12(1\pm\varepsilon)$ split $\QGr^{\Sigma}E$ into two blocks, and the
symmetric groups act on a closed module trivially on one block and through
$\sgn_n$ on the other. Thus, however complicated the representations $E_n$
may be, only the two characters $\mathbf 1$ and $\sgn$ survive
localization, and $\QGr^{\Sigma}E$ carries a canonical $\Z/2$-grading. The
value $\varepsilon=1$ occurs in the situation of Theorem~\ref{reconstr}. The value
$\varepsilon=-1$ occurs as well, the sign character being already visible on
the algebras of Examples~\ref{primery}(7), and the geometry of that block is
of super type. We treat this splitting, and the characters it permits,
in \cite{GG26}.

It is worth mentioning that the same splitting occurs in stable motivic homotopy theory, where 
$SH(k)[1/2]$ splits in two halves $SH(k)_+$ and $SH(k)_{-}$ due to the twist isomorphism 
   $$(\mathbb G_m,1)\wedge (\mathbb G_m,1)\lra{\cong}(\mathbb G_m,1)\wedge (\mathbb G_m,1)$$
permuting factors.

\subsection*{Space, structure sheaf and category are independent invariants}
Classically these three are tied together: Serre's theorem produces
$\Qcoh X$ from the graded ring, and Gabriel's reconstruction~\cite{Ga} recovers a
noetherian scheme $X$ from $\Qcoh X$. In the symmetric setting they come
apart, and this is perhaps the most substantial new phenomenon. 
In~\cite{GG26} we exhibit finitely $\Sigma$-generated algebras $E$,
generated by elements of degree one, for which
\begin{itemize}[leftmargin=1.4em,itemsep=2pt]
\item $\Proj^{\Sigma}E\cong\PP^1_k$ as a space while
 $\cc O_{\Proj^{\Sigma}E}\cong j_*\cc O_{\Aff^1_k}$, so that all
 stalks are nonzero local rings and the ringed space is not a scheme;
\item $\cc O_{\Proj^{\Sigma}E}$ fails to be a sheaf of local rings at
 all, some stalk having two maximal ideals;
\item $(\Proj^{\Sigma}E,\cc O_{\Proj^{\Sigma}E})$ is a classical
 projective scheme $X$ while $\QGr^{\Sigma}E$ is \emph{not} equivalent to
 $\Qcoh X$.
\end{itemize}
Consequently 
the spectral space of Theorem~\ref{d3}, the sheaf of
commutative rings on it, and the monoidal category $\QGr^{\Sigma}E$ carry
strictly different information, and none of the three determines the
others.


\subsection*{Non-proper geometry inside a projective formalism}
For $R$ generated by finitely many elements of $R_1$ over a noetherian
$R_0$, the scheme $\Proj R$ is proper over $\Spec R_0$, so classical
$\Proj$ cannot produce the quasi-coherent category of a non-proper variety.
The symmetric construction can: in~\cite{GG26} we show that for every
quasi-projective $k$-scheme $U$ there is a finitely $\Sigma$-generated
commutative symmetric graded $k$-algebra $E$ with
$\QGr^{\Sigma}E\simeq\Qcoh U$, the space $\Proj^{\Sigma}E$ being a
projective compactification of $U$. Symmetric projective geometry therefore
reaches strictly further than projective geometry does and thus opens exciting new research directions.

\end{document}